\documentclass{paper_class} 

\title{Blow-up vs.~global existence for a\break Fujita-type Heat exchanger system} 

\date{\monthyeardate\today} 

\author[1]{\href{https://www.samueltreton.fr/english/}{Samuel {\sc Tréton}}}
\affil[1]{\footnotesize%
	Univ Rouen Normandie, %
	CNRS UMR 6085, %
	\href{https://lmrs.univ-rouen.fr/en}
	{Laboratoire de Mathématiques Raphaël Salem (LMRS)},\break%
	Avenue de l'Université, %
	BP 12, %
	F-76801 %
	Saint-Étienne-du-Rouvray, %
	France.
	}

\begin{document}

\maketitle

\begin{abstract} 
We analyze a reaction-diffusion system on $\mathbb{R}^{N}$
which models the dispersal of individuals between two exchanging environments for its diffusive component
and incorporates a Fujita-type growth for its reactive component.
The originality of this model lies in the coupling of the equations through diffusion, which, to the best of our knowledge, has not been studied in Fujita-type problems.
We first consider the underlying diffusive problem,
demonstrating that the solutions converge exponentially fast towards those of two uncoupled equations, featuring a dispersive operator that is somehow a combination of Laplacians.
By subsequently adding Fujita-type reaction terms to recover the entire system, we identify the critical exponent that separates systematic blow-up from possible global existence.
\end{abstract}

\vspace{3pt}


\vspace{3pt}

\noindent{\textbf{MSC2020:} %
		\pdftooltip{35K40}{PDEs/Parabolic/Second order parabolic systems},
		\pdftooltip{35B44}{PDEs/Qualitative properties/Blow-up},
		\pdftooltip{35B33}{PDEs/Qualitative properties/Critical exponent},
		\pdftooltip{92D25}{Biology and other natural sciences/Population dynamics (general)}.
		}

\tableofcontents

\section{Introduction}\label{S1_intro}

In this work, we consider the semi-linear Cauchy problem
\renewcommand{\gap}{2.155mm}
\begin{equation}\label{SYS_heat_exchanger}
\left\lbrace \begin{array}{llll}
	\partial_{t}u = c\Delta u - \mu u + \nu v + \hspace{\gap}u^{1+\p}, & \qquad &
	t>0, & x\in\mathbb{R}^{N},\\
	\partial_{t}v = d\Delta v + \mu u - \nu v + \kappa v^{1+\q}, & \qquad &
	t>0, & x\in\mathbb{R}^{N},\\
	\prth{u,v}|_{t=0} = \prth{u_{0},v_{0}}, & \qquad &&
	x\in\mathbb{R}^{N},\\\end{array} \right.
\end{equation}
in any dimension $N\geq 1$. Here the parameters $c,d,\mu,\nu,\p,\q$ are positive constants,
$$
\kappa = 0
\qquad
\text{or}
\qquad 
\kappa = 1,
$$
and the data $u_{0}$ and $v_{0}$ are taken non-both-trivial and non-negative in
$L^{1}\prth{\mathbb{R}^{N}}\cap L^{\infty}\prth{\mathbb{R}^{N}}$.
This system is motivated by biological issues, particularly the field-road reaction-diffusion system
---
see
\reff{SYS_field_road_diff}{%
$\left\lbrace \begin{array}{l}
	\partial_{t}\v = d\Delta \v,\\
	-d\partial_{y}\v|_{y=0} = \mu\u - \nu\v|_{y=0},\\
	\partial_{t}\u = D\Delta \u - \mu\u + \nu\v|_{y=0}.
\end{array}\right.$
}
---
for which no Fujita-type result has been established yet.

It is well-known that
\reff{SYS_heat_exchanger}{%
\renewcommand{\gap}{2.155mm}
$
\left\lbrace \begin{array}{l}
	\partial_{t}u = c\Delta u - \mu u + \nu v + \hspace{\gap}u^{1+\p},\\
	\partial_{t}v = d\Delta v + \mu u - \nu v + \kappa v^{1+\q},\\
	\prth{u,v}|_{t=0} = \prth{u_{0},v_{0}}.
\end{array} \right.
$
}
owns a unique maximal solution
$$
\prth{u,v} \in \prth{C^{1}\prth{\intervalleoo{0}{T},L^{1}\prth{\mathbb{R}^{N}}\cap L^{\infty}\prth{\mathbb{R}^{N}}}}^{2},
$$
such that
$\prth{u, v}$
converges to
$\prth{u_{0}, v_{0}}$
in
$\prth{L^{1}\prth{\mathbb{R}^{N}}}^{2}$
as $t\to 0$.
Furthermore, by parabolic regularity,
$\prth{u,v}$
also belongs to
$\prth{C^{\infty}\prth{\prth{0,T}\times \mathbb{R}^{N}}}^{2}$.
Importantly, the cooperative nature of the system ensures the validity of the comparison principle.
These results are classical and can be found in
\citeee{WeisslerLocal80},
\citeee{FriedmanPartial64}
and
\citeee{FifeComparison81}
for instance.
The main objective of this paper is to determine whether $T$ is finite or not.
If $T=\infty$ we refer to $\prth{u,v}$ as a global solution, while $T<\infty$ enforces
\begin{equation}\label{BLOW_UP_L_infty}
\lim\limits_{t \rightarrow T^{-}}
\Prth{
\vertii{u\prth{t}}{L^{\infty}\prth{\mathbb{R}^{N}}} +
\vertii{v\prth{t}}{L^{\infty}\prth{\mathbb{R}^{N}}}
}
= \infty,
\end{equation}
and we say that $\prth{u,v}$ blows-up in finite time.

\bigskip

Back in $1966$, Fujita conducted the pioneering work
\citeee{FujitaBlowing66}
on the initial value problem
\begin{equation}\label{EQ_Fujita}
\left\lbrace \begin{array}{llll}
	\partial_{t}u = \Delta u + u^{1+\p}, & \qquad &
	t>0, & x\in\mathbb{R}^{N},\\
	u|_{t=0} = u_{0}, & \qquad &&
	x\in\mathbb{R}^{N}.
\end{array} \right.
\end{equation}
He showed that for non-negative and non-trivial datum $u_{0}$, there is a critical exponent $\p_{F} : = 2/N$ (later referred to as the Fujita exponent) that separates systematic blow-up of the solutions from possible global existence.
More precisely, if $0<\p<\p_{F}$, any solution to
\reff{EQ_Fujita}{%
$\left\lbrace \begin{array}{l}
	\partial_{t}u = \Delta u + u^{1+\p},\\
	u|_{t=0} = u_{0}.
\end{array} \right.$}
blows-up in finite time.
Conversely, if $\p>\p_{F}$,
problem
\reff{EQ_Fujita}{%
$\left\lbrace \begin{array}{l}
	\partial_{t}u = \Delta u + u^{1+\p},\\
	u|_{t=0} = u_{0}.
\end{array} \right.$}
admits global solutions
as long as $u_{0}$ is chosen \glmt{sufficiently small}.
Additionally, the global solutions provided by this case
asymptotically vanish in
$L^{\infty}\prth{\mathbb{R}^{N}}$.
The critical case $\p=\p_{F}$ has subsequently been investigated in various works, including
\citeee{HayakawaNonexistence73},
\citeee{KobayashiGrowing77},
\citeee{AronsonMultidimensional78},
\citeee{WeisslerExistence81},
and falls within the systematic blow-up regime.
%

The regime transition that occurs at $1/\p=N/2$ can intuitively be explained by observing that the equation in
\reff{EQ_Fujita}{%
$\left\lbrace \begin{array}{l}
	\partial_{t}u = \Delta u + u^{1+\p},\\
	u|_{t=0} = u_{0}.
\end{array} \right.$}
is a combination of the growth phenomenon resulting from the non-linearity and of the dispersive effect arising from the Laplacian.
The growth of the solutions to
\reff{EQ_Fujita}{%
$\left\lbrace \begin{array}{l}
	\partial_{t}u = \Delta u + u^{1+\p},\\
	u|_{t=0} = u_{0}.
\end{array} \right.$}
is related to the underlying ODE
$U'=U^{1+\p}$
whose solutions have the form
$U\prth{t} = C/\prth{T-t}^{1/\p}$
and blow-up with \textit{algebraic blow-up rate} $1/\p$.
On the other hand, solutions to the Heat equation
$\partial_{t} \u = \Delta \u$
vanish for integrable initial datum,
and the sharp uniform control
$\vertii{\u\prth{t}}{L^{\infty}\prth{\mathbb{R}^{N}}}\leq C/t^{N/2}$
indicates that this vanishing occurs with \textit{algebraic decay rate} $N/2$.
By formally combining the growth and diffusive actions into
\reff{EQ_Fujita}{%
$\left\lbrace \begin{array}{l}
	\partial_{t}u = \Delta u + u^{1+\p},\\
	u|_{t=0} = u_{0}.
\end{array} \right.$},
Fujita's result states that
global solutions can only exist if the diffusion decay rate exceeds the growth blow-up rate, which is $N/2 > 1/\p$.
Otherwise the dispersive effect of the diffusion is never strong enough to prevent the blowing-up driven by the reaction term.

Since the critical case was classified, many works have shown significant interest in
\reff{EQ_Fujita}{%
$\left\lbrace \begin{array}{l}
	\partial_{t}u = \Delta u + u^{1+\p},\\
	u|_{t=0} = u_{0}.
\end{array} \right.$},
either by deepening the established results, or by exploring variations of the problem or additional effects.
A wide range of references can be found in the survey articles by
Levine
\citeee{LevineRole90}
and
Deng and Levine
\citeee{DengRole00},
as well as in the books by
Hu
\citeee{HuBlowup11}
and
Quittner and Souplet
\citeee{QuittnerSuperlinear19}.

The first variation of
\reff{EQ_Fujita}{%
$\left\lbrace \begin{array}{l}
	\partial_{t}u = \Delta u + u^{1+\p},\\
	u|_{t=0} = u_{0}.
\end{array} \right.$}
that takes the form of a system appeared in $1991$ in the work of Escobedo and Herrero
\citeee{EscobedoBoundedness91}.
In their study, they considered
\begin{equation}\label{SYS_Escobedo}
\left\lbrace \begin{array}{llll}
	\partial_{t}u = \Delta u + v^{1+\p}, & \qquad &
	t>0, & x\in\mathbb{R}^{N},\\
	\partial_{t}v = \Delta v + u^{1+\q}, & \qquad &
	t>0, & x\in\mathbb{R}^{N}.\\
\end{array} \right.
\end{equation}
Similarly to the Fujita's problem
\reff{EQ_Fujita}{%
$\left\lbrace \begin{array}{l}
	\partial_{t}u = \Delta u + u^{1+\p},\\
	u|_{t=0} = u_{0}.
\end{array} \right.$},
the system
\reff{SYS_Escobedo}{%
$\left\lbrace \begin{array}{l}
	\partial_{t}u = \Delta u + v^{1+\p},\\
	\partial_{t}v = \Delta v + u^{1+\q}.
\end{array} \right.$
}
may be interpreted as the result of the interplay between a growth phenomenon described by the coupled ODE system
\begin{equation}\label{SYS_Escobedo_growth}
\left\lbrace \begin{array}{lll}
	U' = V^{1+\p}, & \qquad & t>0,\\
	V' = U^{1+\q}, & \qquad & t>0,\\
\end{array} \right.
\end{equation}
and an uncoupled diffusion process governed by
$$
\begin{array}{llll}
	\partial_{t}\u = \Delta \u, & \qquad &
	t>0, & x\in\mathbb{R}^{N},\\[2mm]
	\partial_{t}\v = \Delta \v, & \qquad &
	t>0, & x\in\mathbb{R}^{N}.\\
\end{array}
$$
The results presented in
\citeee{EscobedoBoundedness91}
highlight that the transition
from systematic blow-up from possible global existence
is once again determined by the balance between the growth blow-up rate and the diffusion decay rate.
Specifically, for non-negative and non-both-trivial initial values
$U_{0}, V_{0}$,
both components of the solution $\prth{U,V}$ to
\reff{SYS_Escobedo_growth}{%
$\left\lbrace \begin{array}{l}
	U' = V^{1+\p},\\
	V' = U^{1+\q}.
\end{array} \right.$
}
blows-up in finite time with the blow-up rates
$$
a : = \frac{2+\p}{\p+\q+\p\q}\;\;\, \text{for }U,
\qquad\quad
\text{and}
\qquad\quad
b  : = \frac{2+\q}{\p+\q+\p\q}\;\;\, \text{for }V.
$$
Consequently, for system
\reff{SYS_Escobedo}{%
$\left\lbrace \begin{array}{l}
	\partial_{t}u = \Delta u + v^{1+\p},\\
	\partial_{t}v = \Delta v + u^{1+\q}.
\end{array} \right.$
}
to have global solutions, $N/2$ must be larger than both $a$ and $b$, while systematic blow-up occurs as soon as $a$ or $b$ exceeds $N/2$.

The work of Escobedo and Herrero have been extended in several ways.
First, the introduction of different diffusion coefficients
has been investigated in
\citeee{FilaFujitatype94}
by Fila, Levine and Uda.
Although they reached the same conclusions as in the case of identical diffusions,
it should be noted that considering different diffusion rates significantly increases the complexity of the problem.
This draws attention to one of the key challenges involved in working with system
\reff{SYS_heat_exchanger}{%
\renewcommand{\gap}{2.155mm}
$
\left\lbrace \begin{array}{l}
	\partial_{t}u = c\Delta u - \mu u + \nu v + \hspace{\gap}u^{1+\p},\\
	\partial_{t}v = d\Delta v + \mu u - \nu v + \kappa v^{1+\q},\\
	\prth{u,v}|_{t=0} = \prth{u_{0},v_{0}}.
\end{array} \right.
$
}
which is discussed in more details in Section \ref{S2_main_results}.
More recently, in
\citeee{YangFujitatype15},
problem
\reff{SYS_Escobedo}{%
$\left\lbrace \begin{array}{l}
	\partial_{t}u = \Delta u + v^{1+\p},\\
	\partial_{t}v = \Delta v + u^{1+\q}.
\end{array} \right.$
}
has also been studied by taking some compactly supported non-local diffusion operators instead of the Laplacians.
As for the non-linear effects,
the case of time-weighted reactions
$\prth{f\prth{t}v^{1+\p},g\prth{t}u^{1+\q}}$
was tackled in
\citeee{UdaCritical95},
\citeee{CaoCritical14},
\citeee{CastilloCritical15}, 
and that of space-weighted reactions
$\prth{\verti{x}^{\sigma_{1}}v^{1+\p},\verti{x}^{\sigma_{2}}u^{1+\q}}$
was discussed in
\citeee{MochizukiExistence98}.
An expansion of
\reff{SYS_Escobedo}{%
$\left\lbrace \begin{array}{l}
	\partial_{t}u = \Delta u + v^{1+\p},\\
	\partial_{t}v = \Delta v + u^{1+\q}.
\end{array} \right.$
}
with a \glmt{chain coupling} of more than two unknowns was also handled in
\citeee{RenclawowiczGlobal00}.

Various other types of systems have been developed based on Fujita's problem
\reff{EQ_Fujita}{%
$\left\lbrace \begin{array}{l}
	\partial_{t}u = \Delta u + u^{1+\p},\\
	u|_{t=0} = u_{0}.
\end{array} \right.$}.
We mention for instance systems for which the reaction terms take the form
$\prth{u^{\alpha}v^{\beta},u^{\gamma}v^{\delta}}$
studied in
\citeee{QiGlobal94},
\citeee{EscobedoCritical95},
\citeee{LuGlobal95},
\citeee{SnoussiGlobal02},
\citeee{DicksteinLife07}
or
$\prth{u^{\alpha}+v^{\beta},u^{\gamma}+v^{\delta}}$
considered in
\citeee{CuiGlobal98}, 
\citeee{SoupletOptimal04}, 
\citeee{CastilloCritical15}. 

\bigskip

We would like to stress that in all the semi-linear systems mentioned above, the interaction between the unknowns occurs through the non-linear growth terms. However, this is not the case for problem
\reff{SYS_heat_exchanger}{%
\renewcommand{\gap}{2.155mm}
$
\left\lbrace \begin{array}{l}
	\partial_{t}u = c\Delta u - \mu u + \nu v + \hspace{\gap}u^{1+\p},\\
	\partial_{t}v = d\Delta v + \mu u - \nu v + \kappa v^{1+\q},\\
	\prth{u,v}|_{t=0} = \prth{u_{0},v_{0}}.
\end{array} \right.
$
}
investigated in this paper. Indeed, as for
\reff{EQ_Fujita}{%
$\left\lbrace \begin{array}{l}
	\partial_{t}u = \Delta u + u^{1+\p},\\
	u|_{t=0} = u_{0}.
\end{array} \right.$}
and
\reff{SYS_Escobedo}{%
$\left\lbrace \begin{array}{l}
	\partial_{t}u = \Delta u + v^{1+\p},\\
	\partial_{t}v = \Delta v + u^{1+\q}.
\end{array} \right.$
},
we can extract from
\reff{SYS_heat_exchanger}{%
\renewcommand{\gap}{2.155mm}
$
\left\lbrace \begin{array}{l}
	\partial_{t}u = c\Delta u - \mu u + \nu v + \hspace{\gap}u^{1+\p},\\
	\partial_{t}v = d\Delta v + \mu u - \nu v + \kappa v^{1+\q},\\
	\prth{u,v}|_{t=0} = \prth{u_{0},v_{0}}.
\end{array} \right.
$
}
a \glmt{diffusive component}:
\begin{equation}\label{SYS_heat_exchanger_diff}
\left\lbrace
\begin{array}{llll}
	\partial_{t}\u = c\Delta \u - \mu\u + \nu\v, & \qquad &
	t>0, & x\in\mathbb{R}^{N},\\
	\partial_{t}\v = d\Delta \v + \mu \u - \nu\v, & \qquad &
	t>0, & x\in\mathbb{R}^{N},\\
\end{array}
\right.
\end{equation}
that we call \textit{Heat exchanger} and which contains all the coupling parts of
\reff{SYS_heat_exchanger}{%
\renewcommand{\gap}{2.155mm}
$
\left\lbrace \begin{array}{l}
	\partial_{t}u = c\Delta u - \mu u + \nu v + \hspace{\gap}u^{1+\p},\\
	\partial_{t}v = d\Delta v + \mu u - \nu v + \kappa v^{1+\q},\\
	\prth{u,v}|_{t=0} = \prth{u_{0},v_{0}}.
\end{array} \right.
$
}.
We interpret system
\reff{SYS_heat_exchanger_diff}{%
$\left\lbrace \begin{array}{l}
	\partial_{t}\u = c\Delta \u - \mu \u + \nu\v,\\
	\partial_{t}\v = d\Delta \v + \mu \u - \nu\v.
\end{array}\right.$
}
as a population dynamics model for a single species dispersing on two parallel environments
and switching from one to the other.
Observe that
\reff{SYS_heat_exchanger_diff}{%
$\left\lbrace \begin{array}{l}
	\partial_{t}\u = c\Delta \u - \mu \u + \nu\v,\\
	\partial_{t}\v = d\Delta \v + \mu \u - \nu\v.
\end{array}\right.$
}
preserves the mass of its initial datum over time, that is
\begin{equation}\label{EQ_mass_preseved}
\vertii{\u\prth{t} + \v\prth{t}}{L^{1}\prth{\mathbb{R}^{N}}}
\equiv
\vertii{\u\prth{0} + \v\prth{0}}{L^{1}\prth{\mathbb{R}^{N}}},
\qquad
\text{for all }t>0.
\end{equation}
A detailed description of solutions to
\reff{SYS_heat_exchanger_diff}{%
$\left\lbrace \begin{array}{l}
	\partial_{t}\u = c\Delta \u - \mu \u + \nu\v,\\
	\partial_{t}\v = d\Delta \v + \mu \u - \nu\v.
\end{array} \right.$
}
is given in Theorem
\ref{TH_asymptotic_linear_Heat_exchanger}, Corollary \ref{CORO_decay_linear_Heat_exchanger} and
their surrounding comments.
These results are necessary preliminaries to address the systematic blow-up versus possible global existence for the problem
\reff{SYS_heat_exchanger}{%
\renewcommand{\gap}{2.155mm}
$
\left\lbrace \begin{array}{l}
	\partial_{t}u = c\Delta u - \mu u + \nu v + \hspace{\gap}u^{1+\p},\\
	\partial_{t}v = d\Delta v + \mu u - \nu v + \kappa v^{1+\q},\\
	\prth{u,v}|_{t=0} = \prth{u_{0},v_{0}}.
\end{array} \right.
$
}.

Such a model which exchanges individuals between two domains is not new in the literature.
In particular, the linear Heat exchanger system
\reff{SYS_heat_exchanger_diff}{%
$\left\lbrace \begin{array}{l}
	\partial_{t}\u = c\Delta \u - \mu \u + \nu\v,\\
	\partial_{t}\v = d\Delta \v + \mu \u - \nu\v.
\end{array}\right.$
}
is influenced by the field-road diffusive model
\begin{equation}\label{SYS_field_road_diff}
\left\lbrace
\begin{array}{lllll}
	\partial_{t}\v = d\Delta \v, & \qquad &
	t>0, & x\in\mathbb{R}^{N-1}, & y>0,\\
	-d\partial_{y}\v|_{y=0} = \mu\u - \nu\v|_{y=0}, & \qquad &
	t>0, & x\in\mathbb{R}^{N-1},\\
	\partial_{t}\u = D\Delta \u - \mu\u + \nu\v|_{y=0}, & \qquad &
	t>0, & x\in\mathbb{R}^{N-1},\\
\end{array}
\right.
\end{equation}
where
$\prth{\v,\u}=\prth{\v\prth{t,x,y},\u\prth{t,x}}$,
introduced in $2013$ by Berestycki, Roquejoffre and Rossi in
\citeee{BerestyckiInfluence13}
and whom both fundamental solutions and asymptotic decay rate were recently studied in
\citeee{AlfaroFieldroad23}.
Traditionally, the field-road diffusive model is supplemented with a reaction term $f$ in the field which is represented by the first line of
\reff{SYS_field_road_diff}{%
$\left\lbrace \begin{array}{l}
	\partial_{t}\v = d\Delta \v,\\
	-d\partial_{y}\v|_{y=0} = \mu\u - \nu\v|_{y=0},\\
	\partial_{t}\u = D\Delta \u - \mu\u + \nu\v|_{y=0}.
\end{array}\right.$
}.
As far as we know, no Fujita-type reaction $f$ such as $f\prth{v} = v^{1+\p}$ has been considered in the literature.
In this case
(as with system
\reff{SYS_heat_exchanger}{%
\renewcommand{\gap}{2.155mm}
$
\left\lbrace \begin{array}{l}
	\partial_{t}u = c\Delta u - \mu u + \nu v + \hspace{\gap}u^{1+\p},\\
	\partial_{t}v = d\Delta v + \mu u - \nu v + \kappa v^{1+\q},\\
	\prth{u,v}|_{t=0} = \prth{u_{0},v_{0}}.
\end{array} \right.
$
}),
it is worth noting that
the equations are coupled through the diffusion process.
In this paper, we therefore propose to study problem
\reff{SYS_heat_exchanger}{%
\renewcommand{\gap}{2.155mm}
$
\left\lbrace \begin{array}{l}
	\partial_{t}u = c\Delta u - \mu u + \nu v + \hspace{\gap}u^{1+\p},\\
	\partial_{t}v = d\Delta v + \mu u - \nu v + \kappa v^{1+\q},\\
	\prth{u,v}|_{t=0} = \prth{u_{0},v_{0}}.
\end{array} \right.
$
}
as a first step towards investigating these \glmt{coupled by diffusion} Fujita-type systems.

\bigskip

This work is structured as follows.
In Section \ref{S2_main_results}, we start with an introduction of the notations used throughout.
Subsequently, we present our results on the diffusive system
\reff{SYS_heat_exchanger_diff}{%
$\left\lbrace \begin{array}{l}
	\partial_{t}\u = c\Delta \u - \mu \u + \nu\v,\\
	\partial_{t}\v = d\Delta \v + \mu \u - \nu\v.
\end{array}\right.$
}
and the semi-linear problem
\reff{SYS_heat_exchanger}{%
\renewcommand{\gap}{2.155mm}
$
\left\lbrace \begin{array}{l}
	\partial_{t}u = c\Delta u - \mu u + \nu v + \hspace{\gap}u^{1+\p},\\
	\partial_{t}v = d\Delta v + \mu u - \nu v + \kappa v^{1+\q},\\
	\prth{u,v}|_{t=0} = \prth{u_{0},v_{0}}.
\end{array} \right.
$
}
in Subsection \ref{SS2_1_linear} and {Subsection \ref{SS2_2_non_linear}} respectively.
The remaining sections are devoted to providing the proofs for the statements outlined in Section \ref{S2_main_results}.
Specifically, in Section \ref{S3_asymptotic_behaviour_of_the_diffusion} we prove the results related to linear system
\reff{SYS_heat_exchanger_diff}{%
$\left\lbrace\begin{array}{l}
	\partial_{t}\u = c\Delta \u - \mu \u + \nu\v,\\
	\partial_{t}\v = d\Delta \v + \mu \u - \nu\v.
\end{array}\right.$
}.
In Section \ref{S4_possible_global_existence} we address problem
\reff{SYS_heat_exchanger}{%
\renewcommand{\gap}{2.155mm}
$
\left\lbrace \begin{array}{l}
	\partial_{t}u = c\Delta u - \mu u + \nu v + \hspace{\gap}u^{1+\p},\\
	\partial_{t}v = d\Delta v + \mu u - \nu v + \kappa v^{1+\q},\\
	\prth{u,v}|_{t=0} = \prth{u_{0},v_{0}}.
\end{array} \right.
$
}
when global existence is achievable.
Lastly, we discuss the systematic blow-up case in Section \ref{S5_systematic_blow_up}.

\section{Notations and main results}\label{S2_main_results}


The conventions and notations used in this paper are gathered below.

To avoid confusion, we use bold font
(\eg $\u$, $\v$, $\sigmaSR$, $\deltaSR$)
to denote the solutions to diffusive problems, while we use regular font
(\eg $u$, $v$, $\sigma$, $\delta$)
for the solutions to non-linear problems. Additionally, if
$w = w\prth{t,x}$,
the abbreviation $w\prth{t}$ obviously stands for the function $w\prth{t,\point}$.
We also indicate functions that depend solely on time --- such as ODE solutions --- in capital letters (\eg $U$, $V$, $S$, $D$).

In the whole document the function
$G = G\prth{t,x} : = \prth{4\pi t}^{-N/2}e^{-\verti{x}^{2}/4t}$ denotes the Heat kernel in $\mathbb{R}^{N}$ with unit diffusion rate, and $\ast$ is the classical convolution product.
For any subset $\Omega\subset\mathbb{R}^{N}$, we denote $\partial\Omega$ its topological frontier.
In addition, we use $\pazocal{B}\prth{0,R}$ to denote the ball
$\{x\in\mathbb{R}^{N}\text{ such that }\verti{x}<R\}$,
and 
$\indicatrice{\pazocal{B}\prth{0,R}}$
to represent its indicator function.

For any $f \in L^{1}\prth{\mathbb{R}^{N}}$, the Fourier transform we employ is defined by
$$
{\Fb}\crochets{f}\prth{\xi} =
\hat{f}\prth{\xi} : =
\int_{\mathbb{R}^{N}}^{}
f\prth{x}
e^{-i \xi \cdot x}
dx,
$$
for which holds the inversion formula
$f = \prth{2\pi}^{-N}{\Fb}\circ{\Fb}\crochets{f\prth{-\,\point}}$
as soon as $\hatt{f} \in L^{1}\prth{\mathbb{R}^{N}}$
and the Hausdorff-Young inequalities
\begin{equation}\label{INEQUALITY_Hausdorff_Young}
\vertii{f}{L^{\infty}\prth{\mathbb{R}^{N}}} \leq
\prth{2\pi}^{-N}
\vertii{\hat{f}}{L^{1}\prth{\mathbb{R}^{N}}},
\qquad
\text{and}
\qquad
\vertii{\hat{f}}{L^{\infty}\prth{\mathbb{R}^{N}}} \leq
\vertii{f}{L^{1}\prth{\mathbb{R}^{N}}}.
\end{equation}

We adopt the notation $\pazocal{S}\prth{\mathbb{R}^{N}}$ to represent the Schwarz space of rapidly decreasing functions that consists of smooth functions
$\varphi : \mathbb{R}^{N} \to \mathbb{R}$
that, along with all their derivatives, exhibit faster decay than any polynomial as
$\verti{x}\to\infty$.

\subsection{The linear problem \reff{SYS_heat_exchanger_diff}{$\left\lbrace\begin{array}{l}
	\partial_{t}\u = c\Delta \u - \mu \u + \nu\v,\\
	\partial_{t}\v = d\Delta \v + \mu \u - \nu\v.
\end{array}\right.$
}}\label{SS2_1_linear}

To start with, let us consider the Heat exchanger system
\reff{SYS_heat_exchanger_diff}{%
$\left\lbrace\begin{array}{l}
	\partial_{t}\u = c\Delta \u - \mu \u + \nu\v,\\
	\partial_{t}\v = d\Delta \v + \mu \u - \nu\v.
\end{array}\right.$
}
in the case where all the parameters
$c,d,\mu,\nu$ are equal to $1$.
Then the linear transformation
$\prth{\sigmaSR,\deltaSR} : = \prth{\u+\v,\u-\v}$
enables the uncoupling of the unknowns, and thus,
\begin{equation}\label{SYS_heat_exchanger_diff_same_diff_solutions}
\begin{array}{l}
\u\prth{t} =
\dfrac{1+e^{-2t}}{2}
\Prth{G\prth{t} \ast u_{0}} +
\dfrac{1-e^{-2t}}{2}
\Prth{G\prth{t} \ast v_{0}},\\[4mm]
\v\prth{t} =
\dfrac{1-e^{-2t}}{2}
\Prth{G\prth{t} \ast u_{0}} +
\dfrac{1+e^{-2t}}{2}
\Prth{G\prth{t} \ast v_{0}}.
\end{array}
\end{equation}
From
\reff{SYS_heat_exchanger_diff_same_diff_solutions}{%
$\begin{array}{l}
\u\prth{t} =
\dfrac{1+e^{-2t}}{2}
\Prth{G\prth{t} \ast u_{0}} +
\dfrac{1-e^{-2t}}{2}
\Prth{G\prth{t} \ast v_{0}},\\[4mm]
\v\prth{t} =
\dfrac{1-e^{-2t}}{2}
\Prth{G\prth{t} \ast u_{0}} +
\dfrac{1+e^{-2t}}{2}
\Prth{G\prth{t} \ast v_{0}}.
\end{array}$
},
observe that
$\prth{\u,\v}$
separates into an evanescent part
$\prth{\u_{e},\v_{\!e}}$
which decays exponentially fast, and a persistent part
$\prth{\u_{\infty},\v_{\!\infty}}$,
which is the solution to the uncoupled Heat equations
\begin{equation*}
\begin{array}{llll}
	\partial_{t}\u_{\infty} = \Delta \u_{\infty}, & \qquad &
	t>0, & x\in\mathbb{R}^{N},\\[2mm]
	\partial_{t}\v_{\!\infty} = \Delta \v_{\!\infty}, & \qquad &
	t>0, & x\in\mathbb{R}^{N}.\\
\end{array}
\end{equation*}
both starting from the averaged datum
$\prth{u_{0}+v_{0}}/2$.
In particular, it becomes clear that both $\u$ and $\v$ converge to zero with algebraic decay rate $N/2$.
Our first primary contribution is to demonstrate that a similar phenomenon occurs in the general case where $c,d,\mu,\nu$ may differ.

\begin{theorem}[Asymptotic behavior of the solutions to the diffusive Heat exchanger]\label{TH_asymptotic_linear_Heat_exchanger}
Let
$\prth{\u,\v}$
be the solution to system
\reff{SYS_heat_exchanger_diff}{%
$\left\lbrace\begin{array}{l}
	\partial_{t}\u = c\Delta \u - \mu \u + \nu\v,\\
	\partial_{t}\v = d\Delta \v + \mu \u - \nu\v.
\end{array}\right.$
}
supplemented with
non-negative
initial datum $\prth{u_{0},v_{0}}$,
such that
$u_{0}$, $v_{0}$, $\hatt{u}_{0}$, $\hatt{v}_{\skp\skp\skp 0}$ are all in $L^{1}\prth{\mathbb{R}^{N}}$.
Define, for $\xi\in\mathbb{R}^{N}$, the radial functions
\begin{equation}\label{DEF_r_and_s_preuve_th_linear}
r\prth{\xi} : = \frac{c-d}{2}\verti{\xi}^{2} + \frac{\mu-\nu}{2},
\qquad\qquad
s\prth{\xi} : = \mu\nu + \crochets{r\prth{\xi}}^{2},
\end{equation}
\begin{equation}\label{DEF_L}
L\prth{\xi} : =
\sqrt{s\prth{\xi}} - \Prth{\frac{c+d}{2}\verti{\xi}^{2} + \frac{\mu+\nu}{2}},
\end{equation}
and the dispersal operator
$\mathcal{L} : \pazocal{S}\prth{\mathbb{R}^{N}} \to \pazocal{S}\prth{\mathbb{R}^{N}}$ through the relation
\begin{equation}\label{DEF_cal_L}
{\Fb} \crochets{\mathcal{L}f} = L \times {\Fb} \crochets{f},
\qquad 
\text{for all }f \in \pazocal{S}\prth{\mathbb{R}^{N}}.
\end{equation}
Finally, let
$\prth{\u_{\infty},\v_{\!\infty}}$
be the solution to the uncoupled diffusion equations
\renewcommand{\gap}{-3mm}
\begin{equation}\label{SYS_diff_non_loc_uncoupled}
\begin{array}{lllll}
	\partial_{t}\u_{\infty} &\hspace{\gap} = \mathcal{L}\?\u_{\infty}, & \qquad &
	t>0, & x\in\mathbb{R}^{N},\\[2mm]
	\partial_{t}\v_{\!\infty} &\hspace{\gap} = \mathcal{L}\?\v_{\infty}, & \qquad &
	t>0, & x\in\mathbb{R}^{N},\\
\end{array}
\end{equation}
respectively starting from the data
$u_{\infty,0}$ and $v_{\infty,0}$
defined via their Fourier transforms:
\begin{equation}\label{DEF_u0v0_diff_non_loc_uncoupled}
\begin{matrix} 
\hat{u}_{\infty,0} : =
\dfrac{1}{2}
\crochets{
\Prth{1 - \dfrac{r\prth{\xi}}{\sqrt{s\prth{\xi}}}} \hat{u}_{0} +
\dfrac{\nu}{\sqrt{s\prth{\xi}}} \, \hat{v}_{\skp\skp\skp 0}},\\[6mm]
\hat{v}_{\skp\infty,0} : =
\dfrac{1}{2}
\crochets{\dfrac{\mu}{\sqrt{s\prth{\xi}}} \, \hat{u}_{0} +
\Prth{1 + \dfrac{r\prth{\xi}}{\sqrt{s\prth{\xi}}}} \hat{v}_{\skp\skp\skp 0}}. \end{matrix}
\end{equation}
Then there are two
positive constants $k$ and $k'$ which depend on
$N,c,d,\mu,\nu$
such that
\renewcommand{\gap}{-3mm}
\begin{equation}\label{CONTROL_conv_expo_diff_non_loc_uncoupled}
\begin{array}{ll}
\vertii{\u\prth{t}-\u_{\infty}\prth{t}}{L^{\infty}\prth{\mathbb{R}^{N}}}
\leq
k
&\hspace{\gap}\Prth{\vertii{u_{0}}{L^{1}\prth{\mathbb{R}^{N}}}+\vertii{v_{0}}{L^{1}\prth{\mathbb{R}^{N}}}}
e^{-t \frac{\Prth{\sqrt{\mu}+\sqrt{\nu}\hspace{0.3mm}}^{2}}{2}}, \\[3mm]
\vertii{\v\prth{t}-\v_{\!\infty}\prth{t}}{L^{\infty}\prth{\mathbb{R}^{N}}}
\leq
k'
&\hspace{\gap}\Prth{\vertii{u_{0}}{L^{1}\prth{\mathbb{R}^{N}}}+\vertii{v_{0}}{L^{1}\prth{\mathbb{R}^{N}}}}
e^{-t \frac{\Prth{\sqrt{\mu}+\sqrt{\nu}\hspace{0.3mm}}^{2}}{2}},
\end{array}
\end{equation}
for all $t>1$.
\end{theorem}

\begin{remarks}[on Theorem \ref{TH_asymptotic_linear_Heat_exchanger}]
~\\[-0.7cm]
\begin{itemize}[label=\textbullet , leftmargin=0mm]
\item[]\hspace{-0.5mm}\textbullet\,\textit{The operator $\mathcal{L}$ ---
\reff{DEF_L}{%
$L\prth{\xi} : =
\sqrt{s\prth{\xi}} - \Prth{\frac{c+d}{2}\verti{\xi}^{2} + \frac{\mu+\nu}{2}}.$
}-\reff{DEF_cal_L}{%
${\Fb} \crochets{\mathcal{L}f} = L \times {\Fb} \crochets{f},
\qquad 
\text{for all }f \in \pazocal{S}\prth{\mathbb{R}^{N}}.$
}.}
By observing the behavior of $L$ at high and low frequencies, we can discern similarities between $\mathcal{L}$ and the Laplacian.
Specifically, recalling the identity
\begin{equation*}
{\Fb} \crochets{c \Delta f} = -c\verti{\xi}^{2} \times {\Fb} \crochets{f},
\qquad 
\text{for all }f \in \pazocal{S}\prth{\mathbb{R}^{N}},
\end{equation*} 
an analysis of $L$ near the null frequency and as
$\verti{\xi}$ approaches infinity
reveals that $\mathcal{L}$ behaves like
$\prth{c\nu+d\mu}\prth{\mu+\nu}^{-1}\Delta$
at low frequencies and
$\min\prth{c,d}\Delta$
at high frequencies.
Note also that $L$ becomes independent on the exchange parameters $\mu$ and $\nu$ when the diffusion rates $c$ and $d$ are equal.
In this case, $\mathcal{L} \equiv c\Delta \equiv d\Delta$.

\hspace*{\myindent}Figure \ref{FIG_L} below presents radial profiles of $L$ for various combinations of the parameters
$c,d,\mu,\nu$.
The first line is the simple case $c=d=\mu=\nu=1$, for which
$L\prth{\xi} = -\verti{\xi}^{2}$.
Upon examining the different profiles of $L$ in Figure \ref{FIG_L}, we observe two distinct types of shapes that depend on how $L$ changes its behavior from zero to infinity.
More precisely,
we see
\glmt{soft transitions} in the first and second lines of the table, and \glmt{sharp transitions} in the third and last lines.
These differences in the shapes of the decay may indicate varying degrees of homogeneity/heterogeneity in how $\u$ and $\v$ interact with individuals in their respective environments.
Indeed, recalling that $\mu$ (resp. $\nu$) represents the outgoing exchange rate of $\u$ (resp. $\v$), we can observe that in the sharp transitions cases only, one of the two unknowns has both a large diffusion rate and an outgoing exchange rate less than or equal to the ingoing one.
In this way, the strongest diffuser that primarily conserves its individuals mostly affects the low frequencies of $L$, its contribution quickly vanishing at high frequencies.
\newpage
\refstepcounter{FIGURE}\label{FIG_L}
\begin{center}
\includegraphics[scale=1]{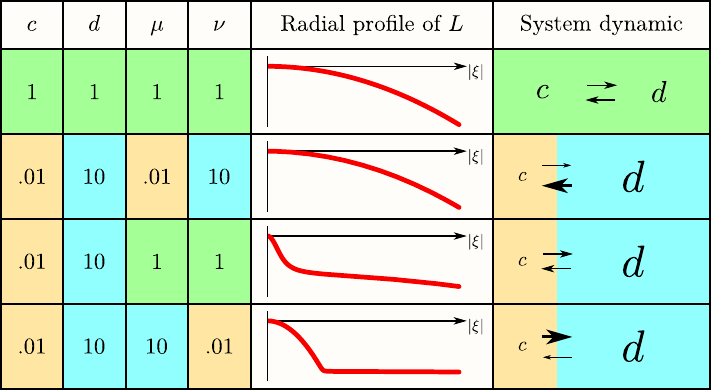}\\[3mm]
\begin{minipage}[c]{14.6079cm}
\textsl{\textbf{Figure \theFIGURE} --- Representation of $L\prth{\xi}$
for different combinations of $c,d,\mu,\nu$. The curves are shown at different scales for readability reasons. The \glmt{soft transition cases} (resp. \glmt{sharp transition cases}) --- see the remarks on $\mathcal{L}$ above --- are displayed in the first and second (resp. third and fourth) lines. The final column provides a visual representation of how $\u$ and $\v$ exchange and spread the individuals.}
\end{minipage}
\end{center}
\item[]\hspace{-0.5mm}\textbullet\,\textit{The persistent data $u_{\infty,0}$ and $v_{\infty,0}$ ---
\reff{DEF_u0v0_diff_non_loc_uncoupled}{%
$\begin{matrix} 
\hat{u}_{\infty,0} : =
\dfrac{1}{2}
\crochets{
\Prth{1 - \dfrac{r}{\sqrt{s}}} \hat{u}_{0} +
\dfrac{\nu}{\sqrt{s}} \, \hat{v}_{\skp\skp\skp 0}},\\[5mm]
\hat{v}_{\skp\infty,0} : =
\dfrac{1}{2}
\crochets{\dfrac{\mu}{\sqrt{s}} \, \hat{u}_{0} +
\Prth{1 + \dfrac{r}{\sqrt{s}}} \hat{v}_{\skp\skp\skp 0}}. \end{matrix}$
}.}
By analyzing
\reff{DEF_r_and_s_preuve_th_linear}{%
$\begin{array}{l}
	r\prth{\xi} : = \frac{c-d}{2}\verti{\xi}^{2} + \frac{\mu-\nu}{2},\\[2mm]
	s\prth{\xi} : = \mu\nu + \crochets{r\prth{\xi}}^{2}.
\end{array}$
},
observe that
$\verti{r/\sqrt{s}}$ and $1/\sqrt{s}$
stay bounded between $0$ and $1$, ensuring the integrability of $\hatt{u}_{\infty,0}$ and $\hatt{v}_{\skp\infty,0}$.
Therefore, $u_{\infty,0}$ and $v_{\infty,0}$ can be defined correctly by
\reff{DEF_u0v0_diff_non_loc_uncoupled}{%
$\begin{matrix} 
\hat{u}_{\infty,0} : =
\dfrac{1}{2}
\crochets{
\Prth{1 - \dfrac{r}{\sqrt{s}}} \hat{u}_{0} +
\dfrac{\nu}{\sqrt{s}} \, \hat{v}_{\skp\skp\skp 0}},\\[5mm]
\hat{v}_{\skp\infty,0} : =
\dfrac{1}{2}
\crochets{\dfrac{\mu}{\sqrt{s}} \, \hat{u}_{0} +
\Prth{1 + \dfrac{r}{\sqrt{s}}} \hat{v}_{\skp\skp\skp 0}}. \end{matrix}$
}
and written using the inversion formula given at the beginning of Section \ref{S2_main_results}.

\hspace*{\myindent}When the diffusion rates $c$ and $d$ are equal, expressions of $u_{\infty,0}$ and $v_{\infty,0}$ become
straightforward linear combinations of $u_{0}$ and $v_{0}$,
and the Fourier transform is no longer required to explain them.
When the diffusion rates are different
--- let us say $c<d$ for simplicity ---
the positive functions $\prth{1+r/\sqrt{s}}/2$, $\mu/2\sqrt{s}$ and $\nu/2\sqrt{s}$ collapse algebraically fast in the high frequencies, acting as low-pass filters.
Similarly, the positive function $\prth{1-r/\sqrt{s}}/2$
acts as a high-pass filter
by remaining below the value $1$ towards which it converges as $\verti{\xi}$ goes to infinity.
Therefore, we can understand the persistent data $u_{\infty,0}$ and $v_{\infty,0}$ as linear combinations of $u_{0}$ and $v_{0}$ that have been altered by high- and low-pass frequency filters.
\end{itemize}
\end{remarks}

\bigskip

We now examine the decay rate of $\prth{\u,\v}$ which is expected to be algebraic.
To reach this,
it suffices to know how $\u_{\infty}\prth{t}$ and $\v_{\!\infty}\prth{t}$ decay in $L^{\infty}\prth{\mathbb{R}^{N}}$,
given the exponential convergence towards the persistent part $\prth{\u_{\infty},\v_{\!\infty}}$ ensured by Theorem
\ref{TH_asymptotic_linear_Heat_exchanger}.
The corollary that follows reveals that this decaying occurs with magnitude $N/2$, which is tightly related to how the operator
$\mathcal{L}$ affects the low frequencies of its argument.

\renewcommand{\gap}{-2.5mm}
\begin{corollary}[Decay rate of the solutions to the diffusive Heat exchanger]\label{CORO_decay_linear_Heat_exchanger} 
Under the hypothesis of Theorem \ref{TH_asymptotic_linear_Heat_exchanger},
there are two 
positive constants $\ell$ and $\ell'$ depending on
$N,c,d,\mu,\nu$,
such that
\renewcommand{\gap}{-3.5mm}
\begin{equation}\label{CONTROL_uv_decay_rate_diff}
\!\!\!\!\!\!\begin{array}{ll}
	\vertii{\u\prth{t}}{L^{\infty}\prth{\mathbb{R}^{N}}}
	\leq
	\ell
	&\hspace{\gap}\Big(
	\vertii{u_{0}}{L^{1}\prth{\mathbb{R}^{N}}}\skp+\skp
	\vertii{v_{0}}{L^{1}\prth{\mathbb{R}^{N}}}\skp+\skp
	\vertii{\?\?\?\?\?\hat{u}_{0}}{L^{1}\prth{\mathbb{R}^{N}}}\skp+\skp
	\vertii{\?\?\?\?\hat{v}_{\skp\skp\skp 0}}{L^{1}\prth{\mathbb{R}^{N}}}\left.\Big)\right/\prth{1+t}^{N/2},\\[2.5mm] 
	\vertii{\v\prth{t}}{L^{\infty}\prth{\mathbb{R}^{N}}}
	\leq
	\ell'
	&\hspace{\gap}\Big(
	\vertii{u_{0}}{L^{1}\prth{\mathbb{R}^{N}}}\skp+\skp
	\vertii{v_{0}}{L^{1}\prth{\mathbb{R}^{N}}}\skp+\skp
	\vertii{\?\?\?\?\?\hat{u}_{0}}{L^{1}\prth{\mathbb{R}^{N}}}\skp+\skp
	\vertii{\?\?\?\?\hat{v}_{\skp\skp\skp 0}}{L^{1}\prth{\mathbb{R}^{N}}}\left.\Big)\right/\prth{1+t}^{N/2},
\end{array}
\end{equation}
for all $t\geq 0$.
\end{corollary}


\renewcommand{\gap}{2.155mm}

\subsection{The semi-linear problem \reff{SYS_heat_exchanger}{$\left\lbrace \begin{array}{l}\partial_{t}u = c\Delta u - \mu u + \nu v + \hspace{\gap}u^{1+\p},\\ \partial_{t}v = d\Delta v + \mu u - \nu v + \kappa v^{1+\q},\\ \prth{u,v}|_{t=0} = \prth{u_{0},v_{0}}. \end{array} \right.
$}}\label{SS2_2_non_linear}

Having examined the behavior of the pure diffusive Heat exchanger
\reff{SYS_heat_exchanger_diff}{%
$\left\lbrace\begin{array}{l}
	\partial_{t}\u = c\Delta \u - \mu \u + \nu\v,\\
	\partial_{t}\v = d\Delta \v + \mu \u - \nu\v.
\end{array}\right.$
}
in the previous subsection,
we are now prepared to address the question of global existence versus blow-up for our main problem
\reff{SYS_heat_exchanger}{%
\renewcommand{\gap}{2.155mm}
$
\left\lbrace \begin{array}{l}
	\partial_{t}u = c\Delta u - \mu u + \nu v + \hspace{\gap}u^{1+\p},\\
	\partial_{t}v = d\Delta v + \mu u - \nu v + \kappa v^{1+\q},\\
	\prth{u,v}|_{t=0} = \prth{u_{0},v_{0}}.
\end{array} \right.
$
}.
Before proceeding with the analysis, it is worth recalling that the comparison principle ensures that, given the same initial datum $\prth{u_{0},v_{0}}$, the solution to
\reff{SYS_heat_exchanger}{%
\renewcommand{\gap}{2.155mm}
$
\left\lbrace \begin{array}{l}
	\partial_{t}u = c\Delta u - \mu u + \nu v + u^{1+\p},\\
	\partial_{t}v = d\Delta v + \mu u - \nu v,\\
	\prth{u,v}|_{t=0} = \prth{u_{0},v_{0}}.
\end{array} \right.
$
}$|_{\kappa=0}$
always remains lower than the solution to
\reff{SYS_heat_exchanger}{%
\renewcommand{\gap}{2.155mm}
$
\left\lbrace \begin{array}{l}
	\partial_{t}u = c\Delta u - \mu u + \nu v + u^{1+\p},\\
	\partial_{t}v = d\Delta v + \mu u - \nu v + v^{1+\q},\\
	\prth{u,v}|_{t=0} = \prth{u_{0},v_{0}}.
\end{array} \right.
$
}$|_{\kappa=1}$.
Consequently, to demonstrate blow-up, it is sufficient to consider
\reff{SYS_heat_exchanger}{%
\renewcommand{\gap}{2.155mm}
$
\left\lbrace \begin{array}{l}
	\partial_{t}u = c\Delta u - \mu u + \nu v + u^{1+\p},\\
	\partial_{t}v = d\Delta v + \mu u - \nu v,\\
	\prth{u,v}|_{t=0} = \prth{u_{0},v_{0}}.
\end{array} \right.
$
}$|_{\kappa=0}$,
while to establish global existence, we only need to examine
\reff{SYS_heat_exchanger}{%
\renewcommand{\gap}{2.155mm}
$
\left\lbrace \begin{array}{l}
	\partial_{t}u = c\Delta u - \mu u + \nu v + u^{1+\p},\\
	\partial_{t}v = d\Delta v + \mu u - \nu v + v^{1+\q},\\
	\prth{u,v}|_{t=0} = \prth{u_{0},v_{0}}.
\end{array} \right.
$
}$|_{\kappa=1}$.

Corollary \ref{CORO_decay_linear_Heat_exchanger} states that the solutions to
\reff{SYS_heat_exchanger_diff}{%
$\left\lbrace\begin{array}{l}
	\partial_{t}\u = c\Delta \u - \mu \u + \nu\v,\\
	\partial_{t}\v = d\Delta \v + \mu \u - \nu\v.
\end{array}\right.$
}
decay uniformly to zero at the algebraic rate $N/2$. With regard to the uncoupled reaction component of
\reff{SYS_heat_exchanger}{%
\renewcommand{\gap}{2.155mm}
$
\left\lbrace \begin{array}{l}
	\partial_{t}u = c\Delta u - \mu u + \nu v + \hspace{\gap}u^{1+\p},\\
	\partial_{t}v = d\Delta v + \mu u - \nu v + \kappa v^{1+\q},\\
	\prth{u,v}|_{t=0} = \prth{u_{0},v_{0}}.
\end{array} \right.
$
},
namely,
\renewcommand{\gap}{-3.5mm}
\begin{equation}\label{SYS_heat_exchanger_growth}
\begin{array}{llll}
	U' = &\hspace{\gap}U^{1+\p}, & \qquad & t>0,\\[2mm]
	V' = \kappa &\hspace{\gap}V^{1+\q}, & \qquad & t>0,
\end{array}
\end{equation}
it is apparent that, given any positive initial data $U_{0}$, $V_{0}$,
\begin{itemize}[label=\textbullet]
	\item $U$ blows-up at the rate $1/\p$ for $\kappa\in\acco{0,1}$, and
	\item $V$ blows-up at the rate $1/\q$ only when $\kappa = 1$.
\end{itemize}
As observed in introduction, we anticipate distinguishing between systematic blow-up and possible global existence by comparing the decay rate of
\reff{SYS_heat_exchanger_diff}{%
$\left\lbrace\begin{array}{l}
	\partial_{t}\u = c\Delta \u - \mu \u + \nu\v,\\
	\partial_{t}\v = d\Delta \v + \mu \u - \nu\v.
\end{array}\right.$
}
with the blow-up rates of\renewcommand{\gap}{-3.5mm}
\reff{SYS_heat_exchanger_growth}{%
$\begin{array}{ll}
	U' = &\hspace{\gap}U^{1+\p},\\
	V' = \kappa &\hspace{\gap}V^{1+\q}.
\end{array}$
}.
To be more specific, the regime transition is expected at
$N/2 = 1/\p$ when $\kappa=0$ and $N/2 = \max\prth{1/\p,1/\q}$ when $\kappa=1$.

These critical exponents can simply be confirmed when the rates $c,d,\mu,\nu$ are identical.
To see this, let us assume $c=d=\mu=\nu=1$, and define
$\prth{\sigma,\delta} : = \prth{u+v,u-v}$.
Observe that the blowing-up of $\sigma$ is equivalent to that of $\prth{u,v}$.
We first examine the solutions to
\reff{SYS_heat_exchanger}{%
\renewcommand{\gap}{2.155mm}
$
\left\lbrace \begin{array}{l}
	\partial_{t}u = c\Delta u - \mu u + \nu v + u^{1+\p},\\
	\partial_{t}v = d\Delta v + \mu u - \nu v + v^{1+\q},\\
	\prth{u,v}|_{t=0} = \prth{u_{0},v_{0}}.
\end{array} \right.
$
}$|_{\kappa=1}$
to explore the possible global existence.
In an attempt to show that $\sigma$ may not blow-up, we have
$$
\partial_{t}\sigma
\quad = \quad
\Delta\sigma + u^{1+\p} + v^{1+\q}
\quad\leq \quad
\Delta\sigma + \sigma^{1+\p} + \sigma^{1+\q}
\quad = : \quad \pazocal{L}\sigma.
$$
Thus $\sigma$ serves as a sub-solution to $\partial_{t}=\pazocal{L}$.
Under the assumption $\min\prth{\p,\q}>2/N$, we can find small enough positive values $\eta$ and $R$ to allow the solution $\bara{\sigma}$ to the Cauchy problem
\begin{equation*}
\left\lbrace \begin{array}{llll}
	\partial_{t}\bara{\sigma} = \Delta \bara{\sigma} + 2\bara{\sigma}^{\,1+\min\prth{\p,\q}}, & \qquad &
	t>0, & x\in\mathbb{R}^{N},\\[1mm]
	\bara{\sigma}|_{t=0} = \eta\indicatrice{\pazocal{B}\prth{0, R}}, & \qquad &&
	x\in\mathbb{R}^{N},
\end{array} \right.
\end{equation*}
to be global and such that $\bara{\sigma}\prth{t}\leq 1$ for any $t>0$.
This can be seen from the classical construction of global super-solutions as the product of a
time-dependent function with the solution to the Heat equation — see
\citeeep{QuittnerSuperlinear19}{130} for instance.
Using this estimation on $\bara{\sigma}$, we write
$$
\partial_{t}\bara{\sigma}
\quad\geq \quad
\Delta\bara{\sigma} +
\bara{\sigma}^{\,1+\min\prth{\p,\q}}+\bara{\sigma}^{\,1+\max\prth{\p,\q}}
\quad =\quad
\pazocal{L}\bara{\sigma},
$$
which indicates that $\bara{\sigma}$ is a global super-solution to $\partial_{t}=\pazocal{L}$.
Now,
$\sigma\leq \bara{\sigma}$
as long as $u_{0}$ and $v_{0}$ are chosen small enough to satisfy
$u_{0}+v_{0}\leq \eta\indicatrice{\pazocal{B}\prth{0, R}}$, and with that condition met, this case is resolved.

For the systematic blow-up
(still with $c=d=\mu=\nu=1$)
we examine
\reff{SYS_heat_exchanger}{%
\renewcommand{\gap}{2.155mm}
$
\left\lbrace \begin{array}{l}
	\partial_{t}u = c\Delta u - \mu u + \nu v + u^{1+\p},\\
	\partial_{t}v = d\Delta v + \mu u - \nu v,\\
	\prth{u,v}|_{t=0} = \prth{u_{0},v_{0}}.
\end{array} \right.
$
}$|_{\kappa=0}$,
and observe that
$$
\partial_{t}\delta
\quad = \quad
\Delta\delta-2\delta + u^{1+\p}
\quad\geq\quad
\Delta\delta-2\delta,
$$
which implies that
$\delta$ remains greater than the sub-solution $e^{-2t}\prth{G\prth{t}\ast\prth{u_{0}-v_{0}}}$, that is positive as long as $u_{0}>v_{0}$.
It is important to note that the assumption $u_{0}>v_{0}$ is not a strict requirement for proving systematic blow-up.
Indeed, we may suppose $u_{0}>0$ (even if we need to shift the time for this to hold) and then work with the solution arising from the initial datum $\prth{u_{0},0}$.
With the established order relation $u>v$ which holds for all times, we can determine now that $\sigma$ systematically blows-up when $\p<2/N$ since
$$
\partial_{t}\sigma 
\quad = \quad
\Delta\sigma + u^{1+\p}
\quad \geq \quad
\Delta\sigma + \prth{u/2 +v/2}^{1+\p}
\quad = \quad
\Delta\sigma + \frac{\sigma^{1+\p}}{2^{1+\p}}.
$$

In those cases of identical rates, the key argument that enables progress concisely stands in
$$
c\Delta u \pm d\Delta v
\quad = \quad
c\Delta\prth{u \pm v}
\quad = \quad
d\Delta\prth{u \pm v}.
$$
However, these equalities break down when the diffusion rates differ, making $c\neq d$ one of the main challenges of the problem. In the general case where $c,d,\mu,\nu$ are arbitrary, we cannot rely on our knowledge about Fujita's original problem
\reff{EQ_Fujita}{%
$\left\lbrace \begin{array}{l}
	\partial_{t}u = \Delta u + u^{1+\p},\\
	u|_{t=0} = u_{0}.
\end{array} \right.$}
to draw conclusions on the lifetime of
solutions to
\reff{SYS_heat_exchanger}{%
\renewcommand{\gap}{2.155mm}
$
\left\lbrace \begin{array}{l}
	\partial_{t}u = c\Delta u - \mu u + \nu v + \hspace{\gap}u^{1+\p},\\
	\partial_{t}v = d\Delta v + \mu u - \nu v + \kappa v^{1+\q},\\
	\prth{u,v}|_{t=0} = \prth{u_{0},v_{0}}.
\end{array} \right.
$
}.
Theorem \ref{TH_possible_global_existence} and Theorem \ref{TH_systematic_BU}
fill this gap and are the main contribution of the present paper.

\begin{theorem}[Possible global existence]\label{TH_possible_global_existence}
Let
\begin{equation}\label{PF_possible_extinction}
\frac{N}{2}>
\left\lbrace \begin{array}{ll}
	\frac{1}{\p} &\quad \text{if }\kappa=0, \\[1mm]
	\max\Prth{\frac{1}{\p}, \frac{1}{\q}} &\quad \text{if }\kappa=1. \\
\end{array} \right.
\end{equation}
Then there is $m_{0}>0$ that depends on $N,c,d,\mu,\nu,\p,\q,\kappa$ such that the solution $\prth{u,v}$ to problem
\reff{SYS_heat_exchanger}{%
\renewcommand{\gap}{2.155mm}
$
\left\lbrace \begin{array}{l}
	\partial_{t}u = c\Delta u - \mu u + \nu v + \hspace{\gap}u^{1+\p},\\
	\partial_{t}v = d\Delta v + \mu u - \nu v + \kappa v^{1+\q},\\
	\prth{u,v}|_{t=0} = \prth{u_{0},v_{0}}.
\end{array} \right.
$
}
with non-negative $\prth{u_{0},v_{0}}$
exists for all times as soon as
\begin{equation}\label{CONTROL_pour_avoir_existence_globale}
m : = \vertii{u_{0}}{L^{1}\prth{\mathbb{R}^{N}}} + 
\vertii{v_{0}}{L^{1}\prth{\mathbb{R}^{N}}} + 
\vertii{\?\?\?\?\?\hat{u}_{0}}{L^{1}\prth{\mathbb{R}^{N}}}\skp+\skp
\vertii{\?\?\?\?\hat{v}_{\skp\skp\skp 0}}{L^{1}\prth{\mathbb{R}^{N}}}
< m_{0}.
\end{equation}
Furthermore, under hypothesis
\reff{CONTROL_pour_avoir_existence_globale}{%
$m : = \vertii{u_{0}}{L^{1}\prth{\mathbb{R}^{N}}} + 
\vertii{v_{0}}{L^{1}\prth{\mathbb{R}^{N}}} + 
\vertii{\?\?\?\?\?\hat{u}_{0}}{L^{1}\prth{\mathbb{R}^{N}}}\skp+\skp
\vertii{\?\?\?\?\hat{v}_{\skp\skp\skp 0}}{L^{1}\prth{\mathbb{R}^{N}}}
< m_{0}.$
},
\begin{equation}\label{CONTROL_solution_globale}
\vertii{u\prth{t}}{L^{\infty}\prth{\mathbb{R}^{N}}}
\leq
\frac{M}{\prth{1+t}^{N/2}}
\qquad
\text{and}
\qquad 
\vertii{v\prth{t}}{L^{\infty}\prth{\mathbb{R}^{N}}}
\leq
\frac{M'}{\prth{1+t}^{N/2}},
\end{equation}
for all $t>0$ and some positive constants $M$ and $M'$
depending on $N,c,d,\mu,\nu,\p,\q,\kappa,\text{and }m$.
\end{theorem}

\begin{theorem}[Systematic blow-up]\label{TH_systematic_BU}
Let
\begin{equation}\label{PF_sytematic_blow_up}
\frac{N}{2}<
\left\lbrace \begin{array}{ll}
	\frac{1}{\p} &\quad \text{if }\kappa=0, \\[1mm]
	\max\Prth{\frac{1}{\p}, \frac{1}{\q}} &\quad \text{if }\kappa=1. \\
\end{array} \right.
\end{equation}
Then any solution $\prth{u,v}$ to problem
\reff{SYS_heat_exchanger}{%
\renewcommand{\gap}{2.155mm}
$
\left\lbrace \begin{array}{l}
	\partial_{t}u = c\Delta u - \mu u + \nu v + \hspace{\gap}u^{1+\p},\\
	\partial_{t}v = d\Delta v + \mu u - \nu v + \kappa v^{1+\q},\\
	\prth{u,v}|_{t=0} = \prth{u_{0},v_{0}}.
\end{array} \right.
$
}
with non-both-trivial and non-negative initial datum
$\prth{u_{0},v_{0}} \in \prth{L^{1}\prth{\mathbb{R}^{N}}\cap L^{\infty}\prth{\mathbb{R}^{N}}}^{2}$
blows-up in finite time.
\end{theorem}

\begin{remark}[on Theorem \ref{TH_systematic_BU}]
~\\[-0.7cm]
\begin{itemize}[label=\textbullet , leftmargin=0mm]
\item[]\hspace{-0.5mm}\textbullet\,\textit{Simultaneous versus non-simultaneous blow-up.}
Based on
\reff{BLOW_UP_L_infty}{$\lim\limits_{t \rightarrow T^{-}}
\Prth{
\vertii{u\prth{t}}{L^{\infty}\prth{\mathbb{R}^{N}}} +
\vertii{v\prth{t}}{L^{\infty}\prth{\mathbb{R}^{N}}}
}
= \infty.$},
it is clear that at least one of the two components of $\prth{u,v}$ becomes unbounded at the blowing time.
Nevertheless, it remains uncertain whether only one component or both components tend to infinity as the blowing-up occurs.
The first scenario is called non-simultaneous blow-up, whereas the second is referred to as simultaneous blow-up.
%
Souplet and Tayachi
have explored
similar issues
in 
\citeee{SoupletOptimal04}
for a Fujita-type system with reactions taking the form of $\prth{u^{\alpha}+v^{\beta},u^{\gamma}+v^{\delta}}$.
In their study, they succeeded to separate systematic simultaneous from possible non-simultaneous blow-up with conditions between $\alpha$ and $\gamma$ and between $\beta$ and $\delta$.
As for problem
\reff{SYS_heat_exchanger}{%
\renewcommand{\gap}{2.155mm}
$
\left\lbrace \begin{array}{l}
	\partial_{t}u = c\Delta u - \mu u + \nu v + \hspace{\gap}u^{1+\p},\\
	\partial_{t}v = d\Delta v + \mu u - \nu v + \kappa v^{1+\q},\\
	\prth{u,v}|_{t=0} = \prth{u_{0},v_{0}}.
\end{array} \right.
$
},
we leave this issue as an open question,
but we nevertheless emphasize that the integrability in time
of the $L^{\infty}\prth{\mathbb{R}^{N}}$-norm of the component that goes to infinity could be a decisive factor in determining the behavior of the other component. 
To illustrate this idea, let us consider the simple case
\reff{SYS_heat_exchanger}{%
\renewcommand{\gap}{2.155mm}
$
\left\lbrace \begin{array}{l}
	\partial_{t}u = c\Delta u - \mu u + \nu v + u^{1+\p},\\
	\partial_{t}v = d\Delta v + \mu u - \nu v,\\
	\prth{u,v}|_{t=0} = \prth{u_{0},v_{0}}.
\end{array} \right.
$
}$|_{\kappa=0}$
with $c=d=\mu=\nu=1$ supplemented with non-negative and identically constant datum $\prth{U_{0},V_{0}}$.
In this case the solution $\prth{U,V}$ does not depend on space and thus solves the ODE system
\renewcommand{\gap}{-3mm}
\begin{equation*}
\left\lbrace \begin{array}{llll}
	U' =  -&\hspace{\gap}U + V + U^{1+\p}, & \qquad & t>0,\\
	V' =   &\hspace{\gap}U - V, & \qquad & t>0.
\end{array} \right.
\end{equation*}
Observe that if we suppose $U_{0}>V_{0}$,
then (by subtracting the unknowns) $U>V$,
so that
$$
\prth{U+V}'
\quad = \quad
U^{1+\p}
\quad \geq \quad
\Prth{U/2 + V/2}^{1+\p}
\quad = \quad
\frac{1}{2^{1+\p}}\prth{U+V}^{1+\p},
$$
and $U+V$ blows-up at some finite time $T$ with at least $U$ becoming unbounded at this time.
Now assume the blow-up rate of $U$ to be of magnitude $a>0$, namely,
\begin{equation*}
\barb{\?\?\?\?\? C}\Big/\prth{T-t}^{a} \quad \leq \quad 
U\prth{t} \quad \leq \quad 
\bara{C}\Big/\prth{T-t}^{a},
\end{equation*}
for any $t\in \intervalleoo{0}{T}$ and some positive constants $\barb{\?\?\?\?\? C} \leq \bara{C}$. Turning to $V$, we have
$$
e^{-t}V_{0} + \barb{\?\?\?\?\? C}\int_{0}^{t} \frac{e^{-\prth{t-s}}}{\prth{T-s}^{a}} ds \quad \leq \quad 
V\prth{t} \quad \leq\quad 
e^{-t}V_{0} + \bara{C}\int_{0}^{t} \frac{e^{-\prth{t-s}}}{\prth{T-s}^{a}} ds.
$$
Thus, the boundedness of $V$ at the blowing time is equivalent to the convergence of the integral $\int_{0}^{T}\prth{T-s}^{-a}ds$, which occurs as soon as $a<1$. This conditions express that $U$ must diverge slowly in some sense if we want $V$ to be able to explode with it. Obviously, this is a formal argument and a rigorous proof would require a more detailed analysis of the system. Nevertheless, this example illustrates the importance of the blow-up rate in determining the behavior of solutions near the blowing-up.
\end{itemize}
\end{remark}

\section{Asymptotic behavior of the Heat exchanger system}\label{S3_asymptotic_behaviour_of_the_diffusion}

This section focuses on the large time behavior of the solution $\prth{\u , \v}$ to the Cauchy problem associated with the linear system
\reff{SYS_heat_exchanger_diff}{%
$\left\lbrace\begin{array}{l}
	\partial_{t}\u = c\Delta \u - \mu \u + \nu\v,\\
	\partial_{t}\v = d\Delta \v + \mu \u - \nu\v.
\end{array}\right.$
}.

We begin by providing the proof of Theorem \ref{TH_asymptotic_linear_Heat_exchanger},
which relies on Fourier analysis of the solution.
By doing so, we are able to separate $\prth{\u , \v}$ into two parts,
$\prth{\u_{\infty},\v_{\!\infty}}$
and
$\prth{\u_{e},\v_{\!e}}$,
where an
$L_{x}^{\infty}\prth{\mathbb{R}^{N}}/L^{1}_{\xi}\prth{\mathbb{R}^{N}}$-analysis
shows that the second part decays exponentially fast.
Finally we evaluate the time derivatives of $\,\hatt{\u}_{\infty}$ and $\,\hatt{\v}_{\!\infty}$
which brings us to the formulation of problem\renewcommand{\gap}{-3mm}
\reff{SYS_diff_non_loc_uncoupled}{%
$\begin{array}{ll}
	\partial_{t}\u_{\infty} &\hspace{\gap} = \mathcal{L}\?\u_{\infty},\\[1mm]
	\partial_{t}\v_{\!\infty} &\hspace{\gap} = \mathcal{L}\?\v_{\infty}.\\
\end{array}$
}-\reff{DEF_u0v0_diff_non_loc_uncoupled}{%
$\begin{matrix} 
\hat{u}_{\infty,0} : =
\dfrac{1}{2}
\crochets{
\Prth{1 - \dfrac{r}{\sqrt{s}}} \hat{u}_{0} +
\dfrac{\nu}{\sqrt{s}} \, \hat{v}_{\skp\skp\skp 0}},\\[5mm]
\hat{v}_{\skp\infty,0} : =
\dfrac{1}{2}
\crochets{\dfrac{\mu}{\sqrt{s}} \, \hat{u}_{0} +
\Prth{1 + \dfrac{r}{\sqrt{s}}} \hat{v}_{\skp\skp\skp 0}}. \end{matrix}$
}.

\bigskip

\begin{proof}[Proof of Theorem \ref{TH_asymptotic_linear_Heat_exchanger}]
Applying the Fourier transform to the equations in
\reff{SYS_heat_exchanger_diff}{%
$\left\lbrace\begin{array}{l}
	\partial_{t}\u = c\Delta \u - \mu \u + \nu\v,\\
	\partial_{t}\v = d\Delta \v + \mu \u - \nu\v.
\end{array}\right.$
}, namely
$$
\partial_{t}\u = c\Delta \u - \mu \u + \nu \v
\qquad
\text{and}
\qquad
\partial_{t}\v = d\Delta \v + \mu \u - \nu \v,
$$
we are led to the ODE system
\begin{equation}\label{ODE_Diffusion_Fourier_side}
\partial_{t} \begin{pmatrix} \,\hat{\u}\, \\ \,\hat{\v}\, \end{pmatrix} =
\underbrace{\begin{pmatrix}
-c\verti{\xi}^{2} - \mu & \nu \\ 
\mu & -d\verti{\xi}^{2} - \nu
\end{pmatrix}}
_{= :  A \prth{\xi}}
\begin{pmatrix} \,\hat{\u}\, \\ \,\hat{\v}\, \end{pmatrix}.
\end{equation}
It may be checked that the two eigenpairs,
$\prth{\lambda_{+}\prth{\xi},e_{+}\prth{\xi}}$ and
$\prth{\lambda_{-}\prth{\xi},e_{-}\prth{\xi}}$,
of $A$ are given by
\begin{equation}\label{DEF_eigen_elements_matrix_A}
\lambda_{\pm} =
-\Prth{\frac{c+d}{2}\verti{\xi}^{2}+\frac{\mu+\nu}{2}} \pm \sqrt{s}
\qquad
\text{and}
\qquad
e{\pm} =
\begin{pmatrix} \nu \\ r \pm \sqrt{s} \end{pmatrix},
\end{equation}
where $r$ and $s$ are defined in
\reff{DEF_r_and_s_preuve_th_linear}{%
$\begin{array}{l}
	r\prth{\xi} : = \frac{c-d}{2}\verti{\xi}^{2} + \frac{\mu-\nu}{2},\\[2mm]
	s\prth{\xi} : = \mu\nu + \crochets{r\prth{\xi}}^{2}.
\end{array}$
}.
As a result, we obtain the following expressions for $\hatt{\u}$ and $\hatt{\v}$:
\renewcommand{\gap}{3mm}
\renewcommand{\gapp}{2mm}
\begin{equation*}\label{EQ_diffusion_Fourier_side}
\begin{pmatrix} \,\hat{\u}\prth{t,\xi} \\[\gapp] \,\hat{\v}\prth{t,\xi} \end{pmatrix}
=
\frac{1}{2}
\begin{pmatrix}
\Prth{1 - \frac{r}{\sqrt{s}}}e^{t\lambda_{+}}+\Prth{1 + \frac{r}{\sqrt{s}}}e^{t\lambda_{-}} &
\frac{\nu}{\sqrt{s}}\Prth{e^{t\lambda_{+}}-e^{t\lambda_{-}}}\\[\gap]
\frac{\mu}{\sqrt{s}}\Prth{e^{t\lambda_{+}}-e^{t\lambda_{-}}} &
\Prth{1 + \frac{r}{\sqrt{s}}}e^{t\lambda_{+}}+\Prth{1 - \frac{r}{\sqrt{s}}}e^{t\lambda_{-}}
\end{pmatrix}
\!\!
\begin{pmatrix} \,\hat{u}_{0}\prth{\xi} \\[\gapp] \,\hat{v}_{\skp\skp\skp 0}\prth{\xi} \end{pmatrix}
\!.
\end{equation*}
We split $\hatt{\u}$ and $\hatt{\v}$ into
\begin{equation}\label{EQ_persistent_part_Fourier_side}
\begin{pmatrix} \,\hat{\u}_{\infty}\prth{t,\xi} \\[\gapp] \,\hat{\v}_{\!\infty}\prth{t,\xi} \end{pmatrix}
: =
\frac{e^{t\lambda_{+}}}{2}
\begin{pmatrix}
1 - \frac{r}{\sqrt{s}}&
\frac{\nu}{\sqrt{s}}\\[\gap]
\frac{\mu}{\sqrt{s}} &
1 + \frac{r}{\sqrt{s}}
\end{pmatrix}
\begin{pmatrix} \,\hat{u}_{0}\prth{\xi} \\[\gapp] \,\hat{v}_{\skp\skp\skp 0}\prth{\xi} \end{pmatrix}
\end{equation}
and
\begin{equation}\label{EQ_evanescent_part_Fourier_side}
\begin{pmatrix} \,\hat{\u}_{e}\prth{t,\xi} \\[\gapp] \,\hat{\v}_{\!e}\prth{t,\xi} \end{pmatrix}
: =
\frac{e^{t\lambda_{-}}}{2}
\begin{pmatrix}
1 + \frac{r}{\sqrt{s}}&
-\frac{\nu}{\sqrt{s}}\\[\gap]
-\frac{\mu}{\sqrt{s}} &
1 - \frac{r}{\sqrt{s}}
\end{pmatrix}
\begin{pmatrix} \,\hat{u}_{0}\prth{\xi} \\[\gapp] \,\hat{v}_{\skp\skp\skp 0}\prth{\xi} \end{pmatrix},
\end{equation}
so that
$$
\hat{\u}=\hat{\u}_{\infty}+\hat{\u}_{e}
\qquad
\text{and}
\qquad
\hat{\v}=\hat{\v}_{\!\infty}+\hat{\v}_{\!e},
$$
and the boundedness of the matrices in\renewcommand{\gap}{3mm}\renewcommand{\gapp}{2mm}
\reff{EQ_persistent_part_Fourier_side}{%
$\begin{pmatrix} \,\hat{\u}_{\infty}\prth{t,\xi} \\[\gapp] \,\hat{\v}_{\!\infty}\prth{t,\xi} \end{pmatrix}
: =
\frac{e^{t\lambda_{+}}}{2}
\begin{pmatrix}
1 - \frac{r}{\sqrt{s}}&
\frac{\nu}{\sqrt{s}}\\[\gap]
\frac{\mu}{\sqrt{s}} &
1 + \frac{r}{\sqrt{s}}
\end{pmatrix}
\begin{pmatrix} \,\hat{u}_{0}\prth{\xi} \\[\gapp] \,\hat{v}_{\skp\skp\skp 0}\prth{\xi} \end{pmatrix}.$
}
and\renewcommand{\gap}{3mm}\renewcommand{\gapp}{2mm}
\reff{EQ_evanescent_part_Fourier_side}{%
$\begin{pmatrix} \,\hat{\u}_{e}\prth{t,\xi} \\[\gapp] \,\hat{\v}_{\!e}\prth{t,\xi} \end{pmatrix}
: =
\frac{e^{t\lambda_{-}}}{2}
\begin{pmatrix}
1 + \frac{r}{\sqrt{s}}&
-\frac{\nu}{\sqrt{s}}\\[\gap]
-\frac{\mu}{\sqrt{s}} &
1 - \frac{r}{\sqrt{s}}
\end{pmatrix}
\begin{pmatrix} \,\hat{u}_{0}\prth{\xi} \\[\gapp] \,\hat{v}_{\skp\skp\skp 0}\prth{\xi} \end{pmatrix}.$
}
provides the necessary integrability to go back into the spatial domain.
As a result, we can write
$$
\u=\u_{\infty}+\u_{e}
\qquad
\text{and}
\qquad
\v=\v_{\!\infty}+\v_{\!e}.
$$

The remainder of the proof consists in demonstrating that the evanescent parts
$\u_{e}\prth{t}$ and $\v_{\!e}\prth{t}$
decay exponentially fast in
$L^{\infty}\prth{\mathbb{R}^{N}}$.
Following this, we establish that
$\prth{\u_{\infty},\v_{\!\infty}}$
indeed serves as the solution to the Cauchy problem\renewcommand{\gap}{-3mm}
\reff{SYS_diff_non_loc_uncoupled}{%
$\begin{array}{ll}
	\partial_{t}\u_{\infty} &\hspace{\gap} = \mathcal{L}\?\u_{\infty},\\[1mm]
	\partial_{t}\v_{\!\infty} &\hspace{\gap} = \mathcal{L}\?\v_{\infty}.\\
\end{array}$
}-\reff{DEF_u0v0_diff_non_loc_uncoupled}{%
$\begin{matrix} 
\hat{u}_{\infty,0} : =
\dfrac{1}{2}
\crochets{
\Prth{1 - \dfrac{r}{\sqrt{s}}} \hat{u}_{0} +
\dfrac{\nu}{\sqrt{s}} \, \hat{v}_{\skp\skp\skp 0}},\\[5mm]
\hat{v}_{\skp\infty,0} : =
\dfrac{1}{2}
\crochets{\dfrac{\mu}{\sqrt{s}} \, \hat{u}_{0} +
\Prth{1 + \dfrac{r}{\sqrt{s}}} \hat{v}_{\skp\skp\skp 0}}. \end{matrix}$
}.

\bigskip

\noindent\hspace{-0.5mm}\textbullet\,\textit{Vanishing of $\prth{\u_{e}\prth{t},\v_{\!e}\prth{t}}$.}
Thanks to the Hausdorff-Young inequalities
\reff{INEQUALITY_Hausdorff_Young}{%
$
\begin{array}{l}
	\vertii{f}{L^{\infty}\prth{\mathbb{R}^{N}}} \leq
\prth{2\pi}^{-N}
\vertii{\hat{f}}{L^{1}\prth{\mathbb{R}^{N}}}, \\[2mm]
	\vertii{\hat{f}}{L^{\infty}\prth{\mathbb{R}^{N}}} \leq
\vertii{f}{L^{1}\prth{\mathbb{R}^{N}}}. 
\end{array}
$
},
controlling
$\hatt{\u}_{e}\prth{t}$ and $\hatt{\v}_{\!e}\prth{t}$
in $L^{1}\prth{\mathbb{R}^{N}}$ enables us to obtain estimates on 
$\u_{e}\prth{t}$ and $\v_{\!e}\prth{t}$
in $L^{\infty}\prth{\mathbb{R}^{N}}$.
To accomplish this, we first analyze the components of\renewcommand{\gap}{3mm}\renewcommand{\gapp}{2mm}
\reff{EQ_evanescent_part_Fourier_side}{%
$\begin{pmatrix} \,\hat{\u}_{e}\prth{t,\xi} \\[\gapp] \,\hat{\v}_{\!e}\prth{t,\xi} \end{pmatrix}
: =
\frac{e^{t\lambda_{-}}}{2}
\begin{pmatrix}
1 + \frac{r}{\sqrt{s}}&
-\frac{\nu}{\sqrt{s}}\\[\gap]
-\frac{\mu}{\sqrt{s}} &
1 - \frac{r}{\sqrt{s}}
\end{pmatrix}
\begin{pmatrix} \,\hat{u}_{0}\prth{\xi} \\[\gapp] \,\hat{v}_{\skp\skp\skp 0}\prth{\xi} \end{pmatrix}.$
}
which result in the following estimates for all $\xi\in\mathbb{R}^{N}$:
\begin{equation}\label{CONTROL_lambda_moins}
\lambda_{-} \leq
-\Prth{\frac{c+d}{2}\verti{\xi}^{2} + \frac{\Prth{\sqrt{\mu}+\sqrt{\nu}\,}^{2}}{2}},
\end{equation}
\begin{equation}\label{CONTROL_elts_evanescent_part}
\verti{1\pm\frac{r}{\sqrt{s}}} \leq 2,
\qquad
\qquad
\frac{\nu}{\sqrt{s}} \leq \sqrt{\frac{\nu}{\mu}},
\qquad
\qquad
\frac{\mu}{\sqrt{s}} \leq \sqrt{\frac{\mu}{\nu}},
\end{equation}
and using
\reff{INEQUALITY_Hausdorff_Young}{%
$
\begin{array}{l}
	\vertii{f}{L^{\infty}\prth{\mathbb{R}^{N}}} \leq
\prth{2\pi}^{-N}
\vertii{\hat{f}}{L^{1}\prth{\mathbb{R}^{N}}}, \\[2mm]
	\vertii{\hat{f}}{L^{\infty}\prth{\mathbb{R}^{N}}} \leq
\vertii{f}{L^{1}\prth{\mathbb{R}^{N}}}. 
\end{array}
$
},
\begin{equation}\label{CONTROL_hat_u0_v0}
\Big\vert\,\hat{u}_{0}\prth{\xi}\Big\vert\leq
\vertii{u_{0}}{L^{1}\prth{\mathbb{R}^{N}}}
\qquad
\text{and}
\qquad
\Big\vert\,\hat{v}_{\skp\skp\skp 0}\prth{\xi}\Big\vert\leq 
\vertii{v_{0}}{L^{1}\prth{\mathbb{R}^{N}}}.
\end{equation}
Then, by considering the expression for $\hatt{\u}_{e}\prth{t}$ in\renewcommand{\gap}{3mm}\renewcommand{\gapp}{2mm}
\reff{EQ_evanescent_part_Fourier_side}{%
$\begin{pmatrix} \,\hat{\u}_{e}\prth{t,\xi} \\[\gapp] \,\hat{\v}_{\!e}\prth{t,\xi} \end{pmatrix}
: =
\frac{e^{t\lambda_{-}}}{2}
\begin{pmatrix}
1 + \frac{r}{\sqrt{s}}&
-\frac{\nu}{\sqrt{s}}\\[\gap]
-\frac{\mu}{\sqrt{s}} &
1 - \frac{r}{\sqrt{s}}
\end{pmatrix}
\begin{pmatrix} \,\hat{u}_{0}\prth{\xi} \\[\gapp] \,\hat{v}_{\skp\skp\skp 0}\prth{\xi} \end{pmatrix}.$
},
the estimation of its $L^{1}$-norm gives, for $t>1$,
\begin{align}
 	\vertii{\,\hat{\u}_{e}\prth{t}}{L^{1}\prth{\mathbb{R}^{N}}}
	&= \frac{1}{2}\int_{\mathbb{R}^{N}}^{}\verti{\Prth{1+\frac{r}{\sqrt{s}}}\hat{u}_{0}\prth{\xi} - \frac{\nu}{\sqrt{s}}\,\hat{v}_{\skp\skp\skp 0}\prth{\xi}}e^{t\lambda_{-}} \, d\xi \nonumber\\
 	&\leq \frac{1}{2}\Prth{2\vertii{u_{0}}{L^{1}\prth{\mathbb{R}^{N}}}+\sqrt{\frac{\nu}{\mu}}\vertii{v_{0}}{L^{1}\prth{\mathbb{R}^{N}}}}e^{-t \frac{\Prth{\sqrt{\mu}+\sqrt{\nu}\hspace{0.3mm}}^{2}}{2}}\int_{\mathbb{R}^{N}}^{}e^{-t\frac{c+d}{2}\verti{\xi}^{2}}d\xi \nonumber\\[2mm]
 	&\leq
 	\max\Prth{1,\sqrt{\frac{\nu}{4\mu}}\,}
 	\Prth{\frac{2\pi}{c+d}}^{N/2}
 	\Prth{\vertii{u_{0}}{L^{1}\prth{\mathbb{R}^{N}}}+\vertii{v_{0}}{L^{1}\prth{\mathbb{R}^{N}}}}
 	e^{-t \frac{\Prth{\sqrt{\mu}+\sqrt{\nu}\hspace{0.3mm}}^{2}}{2}},\label{CONTROL_hat_ue_L_1}
\end{align}
where
\reff{CONTROL_lambda_moins}{%
$\lambda_{-} \leq
-\Prth{\frac{c+d}{2}\verti{\xi}^{2} + \frac{\Prth{\sqrt{\mu}+\sqrt{\nu}\,}^{2}}{2}}.$
},
\reff{CONTROL_elts_evanescent_part}{%
$\verti{1\pm\frac{r}{\sqrt{s}}} \leq 2,
\qquad
\frac{\nu}{\sqrt{s}} \leq \sqrt{\frac{\nu}{\mu}},
\qquad
\frac{\mu}{\sqrt{s}} \leq \sqrt{\frac{\mu}{\nu}}.$
}
and
\reff{CONTROL_hat_u0_v0}{%
$\Big\vert\,\hat{u}_{0}\prth{\xi}\Big\vert\leq
\vertii{u_{0}}{L^{1}\prth{\mathbb{R}^{N}}},
\qquad
\Big\vert\,\hat{v}_{\skp\skp\skp 0}\prth{\xi}\Big\vert\leq 
\vertii{v_{0}}{L^{1}\prth{\mathbb{R}^{N}}}.$
}
have been used to go from the first to the second line.
Performing a similar calculation for $\hatt{\v}_{\!e}\prth{t}$ leads to
\begin{equation}\label{CONTROL_hat_ve_L_1}
\vertii{\,\hat{\v}_{\!e}\prth{t}}{L^{1}\prth{\mathbb{R}^{N}}}\leq
\max\Prth{1,\sqrt{\frac{\mu}{4\nu}}\,}
\Prth{\frac{2\pi}{c+d}}^{N/2}
\Prth{\vertii{u_{0}}{L^{1}\prth{\mathbb{R}^{N}}}+\vertii{v_{0}}{L^{1}\prth{\mathbb{R}^{N}}}}e^{-t \frac{\Prth{\sqrt{\mu}+\sqrt{\nu}\hspace{0.3mm}}^{2}}{2}},
\end{equation}
so that, using again
\reff{INEQUALITY_Hausdorff_Young}{%
$
\begin{array}{l}
	\vertii{f}{L^{\infty}\prth{\mathbb{R}^{N}}} \leq
\prth{2\pi}^{-N}
\vertii{\hat{f}}{L^{1}\prth{\mathbb{R}^{N}}}, \\[2mm]
	\vertii{\hat{f}}{L^{\infty}\prth{\mathbb{R}^{N}}} \leq
\vertii{f}{L^{1}\prth{\mathbb{R}^{N}}}. 
\end{array}
$
},
we can retrieve\renewcommand{\gap}{-3mm}
\reff{CONTROL_conv_expo_diff_non_loc_uncoupled}{%
$\begin{array}{ll}
\vertii{\u\prth{t}-\u_{\infty}\prth{t}}{L^{\infty}\prth{\mathbb{R}^{N}}}
\leq
k
&\hspace{\gap}\Prth{\vertii{u_{0}}{L^{1}\prth{\mathbb{R}^{N}}}+\vertii{v_{0}}{L^{1}\prth{\mathbb{R}^{N}}}}
e^{-t \frac{\Prth{\sqrt{\mu}+\sqrt{\nu}\hspace{0.3mm}}^{2}}{2}}, \\[2mm]
\vertii{\v\prth{t}-\v_{\!\infty}\prth{t}}{L^{\infty}\prth{\mathbb{R}^{N}}}
\leq
k'
&\hspace{\gap}\Prth{\vertii{u_{0}}{L^{1}\prth{\mathbb{R}^{N}}}+\vertii{v_{0}}{L^{1}\prth{\mathbb{R}^{N}}}}
e^{-t \frac{\Prth{\sqrt{\mu}+\sqrt{\nu}\hspace{0.3mm}}^{2}}{2}}, \\[2mm]
\text{for all }t>1.
\end{array}$
}
from
\reff{CONTROL_hat_ue_L_1}{%
$\vertii{\,\hat{\u}_{e}\prth{t}}{L^{1}\prth{\mathbb{R}^{N}}}
\leq
\max\Prth{1,\sqrt{\frac{\nu}{4\mu}}\,}
\Prth{\frac{2\pi}{c+d}}^{N/2}
\Prth{\vertii{u_{0}}{L^{1}\prth{\mathbb{R}^{N}}}+\vertii{v_{0}}{L^{1}\prth{\mathbb{R}^{N}}}}
e^{-t \frac{\Prth{\sqrt{\mu}+\sqrt{\nu}\hspace{0.3mm}}^{2}}{2}}.$
}
and
\reff{CONTROL_hat_ve_L_1}{%
$\vertii{\,\hat{\v}_{\!e}\prth{t}}{L^{1}\prth{\mathbb{R}^{N}}}\leq
\max\Prth{1,\sqrt{\frac{\mu}{4\nu}}\,}
\Prth{\frac{2\pi}{c+d}}^{N/2}
\Prth{\vertii{u_{0}}{L^{1}\prth{\mathbb{R}^{N}}}+\vertii{v_{0}}{L^{1}\prth{\mathbb{R}^{N}}}}e^{-t \frac{\Prth{\sqrt{\mu}+\sqrt{\nu}\hspace{0.3mm}}^{2}}{2}}.$
}
with
$$
k : =
\frac{\max\Big(1,\sqrt{\frac{\nu}{4\mu}}\?\?\?\Big)}{\Prth{2\pi\prth{c+d}}^{N/2}}
\qquad 
\text{and}
\qquad 
k' : =
\frac{\max\Prth{1,\sqrt{\frac{\mu}{4\nu}}\?\?\?}}{\Prth{2\pi\prth{c+d}}^{N/2}}.
$$

\bigskip

\noindent\hspace{-0.5mm}\textbullet\,\textit{Problem satisfied by $\prth{\u_{\infty},\v_{\!\infty}}$.}
Taking first $t=0$ in\renewcommand{\gap}{3mm}\renewcommand{\gapp}{2mm}
\reff{EQ_persistent_part_Fourier_side}{%
$\begin{pmatrix} \,\hat{\u}_{\infty}\prth{t,\xi} \\[\gapp] \,\hat{\v}_{\!\infty}\prth{t,\xi} \end{pmatrix}
: =
\frac{e^{t\lambda_{+}}}{2}
\begin{pmatrix}
1 - \frac{r}{\sqrt{s}}&
\frac{\nu}{\sqrt{s}}\\[\gap]
\frac{\mu}{\sqrt{s}} &
1 + \frac{r}{\sqrt{s}}
\end{pmatrix}
\begin{pmatrix} \,\hat{u}_{0}\prth{\xi} \\[\gapp] \,\hat{v}_{\skp\skp\skp 0}\prth{\xi} \end{pmatrix}.$
}
straightly recovers the initial datum
$\prth{\?\?\?\?\hatt{u}_{\infty,0},\hatt{v}_{\skp\infty,0}}$
in
\reff{DEF_u0v0_diff_non_loc_uncoupled}{%
$\begin{matrix} 
\hat{u}_{\infty,0} : =
\dfrac{1}{2}
\crochets{
\Prth{1 - \dfrac{r}{\sqrt{s}}} \hat{u}_{0} +
\dfrac{\nu}{\sqrt{s}} \, \hat{v}_{\skp\skp\skp 0}},\\[5mm]
\hat{v}_{\skp\infty,0} : =
\dfrac{1}{2}
\crochets{\dfrac{\mu}{\sqrt{s}} \, \hat{u}_{0} +
\Prth{1 + \dfrac{r}{\sqrt{s}}} \hat{v}_{\skp\skp\skp 0}}. \end{matrix}$
}.
Then, notice that $\lambda_{+}$ defined in
\reff{DEF_eigen_elements_matrix_A}{%
$\lambda_{\pm} =
-\Prth{\frac{c+d}{2}\verti{\xi}^{2}+\frac{\mu+\nu}{2}} \pm \sqrt{s}.$
}
and $L$ defined in
\reff{DEF_L}{%
$L\prth{\xi} : =
\sqrt{s\prth{\xi}} - \Prth{\frac{c+d}{2}\verti{\xi}^{2} + \frac{\mu+\nu}{2}}.$
}
are actually the same functions.
Hence,
by working on $\partial_{t}\u_{\infty}$ from the frequency domain, we have
$$
{\Fb}\crochets{\partial_{t}\u_{\infty}} \,=\,
\partial_{t}\, \hat{\u}_{\infty}\,=\,
\lambda_{+} \!\skp\times \hat{\u}_{\infty}\,=\,
L \times \hat{\u}_{\infty}\,=\,
{\Fb}\crochets{\mathcal{L}\?\u_{\infty}},
$$
where the last equality precisely comes from the definition of the operator $\mathcal{L}$ in \reff{DEF_cal_L}{%
${\Fb} \crochets{\mathcal{L}f} = L \times {\Fb} \crochets{f},
\qquad 
\text{for all }f \in \pazocal{S}\prth{\mathbb{R}^{N}}.$
}.
The same calculation for $\partial_{t}\v_{\!\infty}$ gives
$$
{\Fb}\crochets{\partial_{t}\v_{\!\infty}} =
{\Fb}\crochets{\mathcal{L}\?\v_{\!\infty}},
$$
and thus,
$\prth{\u_{\infty},\v_{\!\infty}}$
solves\renewcommand{\gap}{-3mm}
\reff{SYS_diff_non_loc_uncoupled}{%
$\begin{array}{ll}
	\partial_{t}\u_{\infty} &\hspace{\gap} = \mathcal{L}\?\u_{\infty},\\[1mm]
	\partial_{t}\v_{\!\infty} &\hspace{\gap} = \mathcal{L}\?\v_{\infty}.\\
\end{array}$
}.
\end{proof}


\bigskip
\medskip

We now proceed with the proof of Corollary
\ref{CORO_decay_linear_Heat_exchanger}
regarding the decay rate.
We start by obtaining uniform controls on
$\u\prth{t}$ and $\v\prth{t}$ for $t>1$, specifically,
\renewcommand{\gap}{-3.5mm}
\begin{equation}\label{CONTROL_heristic_uv_decay_rate_diff}
\begin{array}{ll}
	\vertii{\u\prth{t}}{L^{\infty}\prth{\mathbb{R}^{N}}}
	\leq
	\widetilde{\ell}
	&\hspace{\gap}\Big(
	\vertii{u_{0}}{L^{1}\prth{\mathbb{R}^{N}}}+
	\vertii{v_{0}}{L^{1}\prth{\mathbb{R}^{N}}}+
	\vertii{\?\?\?\?\?\hat{u}_{0}}{L^{1}\prth{\mathbb{R}^{N}}}+
	\vertii{\?\?\?\?\hat{v}_{\skp\skp\skp 0}}{L^{1}\prth{\mathbb{R}^{N}}}\left.\Big)\right/t^{N/2}, \\[2.5mm] 
	\vertii{\v\prth{t}}{L^{\infty}\prth{\mathbb{R}^{N}}}
	\leq
	\widetilde{\ell'}
	&\hspace{\gap}\Big(
	\vertii{u_{0}}{L^{1}\prth{\mathbb{R}^{N}}}+
	\vertii{v_{0}}{L^{1}\prth{\mathbb{R}^{N}}}+
	\vertii{\?\?\?\?\?\hat{u}_{0}}{L^{1}\prth{\mathbb{R}^{N}}}+
	\vertii{\?\?\?\?\hat{v}_{\skp\skp\skp 0}}{L^{1}\prth{\mathbb{R}^{N}}}\left.\Big)\right/t^{N/2}.
\end{array}
\end{equation}
The approach used to derive\renewcommand{\gap}{-3.5mm}
\reff{CONTROL_heristic_uv_decay_rate_diff}{%
$\begin{array}{ll}
	\vertii{\u\prth{t}}{L^{\infty}\prth{\mathbb{R}^{N}}}
	\leq
	\widetilde{\ell}
	&\hspace{\gap}\Big(
	\vertii{u_{0}}{L^{1}\prth{\mathbb{R}^{N}}}+
	\vertii{v_{0}}{L^{1}\prth{\mathbb{R}^{N}}}+
	\vertii{\?\?\?\?\?\hat{u}_{0}}{L^{1}\prth{\mathbb{R}^{N}}}+
	\vertii{\?\?\?\?\hat{v}_{\skp\skp\skp 0}}{L^{1}\prth{\mathbb{R}^{N}}}\left.\Big)\right/t^{N/2}, \\[2.5mm] 
	\vertii{\v\prth{t}}{L^{\infty}\prth{\mathbb{R}^{N}}}
	\leq
	\widetilde{\ell'}
	&\hspace{\gap}\Big(
	\vertii{u_{0}}{L^{1}\prth{\mathbb{R}^{N}}}+
	\vertii{v_{0}}{L^{1}\prth{\mathbb{R}^{N}}}+
	\vertii{\?\?\?\?\?\hat{u}_{0}}{L^{1}\prth{\mathbb{R}^{N}}}+
	\vertii{\?\?\?\?\hat{v}_{\skp\skp\skp 0}}{L^{1}\prth{\mathbb{R}^{N}}}\left.\Big)\right/t^{N/2}.
\end{array}$
}
involves splitting high versus low frequencies of the solution and
recognizing that only the low frequencies contribute algebraically at large times
---
refer to
\citeee{ChasseigneAsymptotic06}
and
\citeee{AlfaroFujita17}
for related phenomena.
Given this observation, it is sufficient to note that $\u\prth{t}$ and $\v\prth{t}$ are bounded for $t\leq 1$ to establish\renewcommand{\gap}{-3.5mm}
\reff{CONTROL_uv_decay_rate_diff}{%
$\begin{array}{ll}
	\vertii{\u\prth{t}}{L^{\infty}\prth{\mathbb{R}^{N}}}
	\leq
	\ell
	&\hspace{\gap}\Big(
	\vertii{u_{0}}{L^{1}\prth{\mathbb{R}^{N}}}\skp+\skp
	\vertii{v_{0}}{L^{1}\prth{\mathbb{R}^{N}}}\skp+\skp
	\vertii{\?\?\?\?\?\hat{u}_{0}}{L^{1}\prth{\mathbb{R}^{N}}}\skp+\skp
	\vertii{\?\?\?\?\hat{v}_{\skp\skp\skp 0}}{L^{1}\prth{\mathbb{R}^{N}}}\left.\Big)\right/\prth{1+t}^{N/2},\\[2.5mm] 
	\vertii{\v\prth{t}}{L^{\infty}\prth{\mathbb{R}^{N}}}
	\leq
	\ell'
	&\hspace{\gap}\Big(
	\vertii{u_{0}}{L^{1}\prth{\mathbb{R}^{N}}}\skp+\skp
	\vertii{v_{0}}{L^{1}\prth{\mathbb{R}^{N}}}\skp+\skp
	\vertii{\?\?\?\?\?\hat{u}_{0}}{L^{1}\prth{\mathbb{R}^{N}}}\skp+\skp
	\vertii{\?\?\?\?\hat{v}_{\skp\skp\skp 0}}{L^{1}\prth{\mathbb{R}^{N}}}\left.\Big)\right/\prth{1+t}^{N/2},\\[2.5mm] 
	\text{for all }t>0.&~
\end{array}$
}
for all times.

\bigskip

\begin{proof}[Proof of Corollary \ref{CORO_decay_linear_Heat_exchanger}]
Recall that
$\u_{e}=\u-\u_{\infty}$ and $\v_{\!e}=\v-\v_{\!\infty}$. Referring to\renewcommand{\gap}{-3mm}
\reff{CONTROL_conv_expo_diff_non_loc_uncoupled}{%
$\begin{array}{ll}
\vertii{\u\prth{t}-\u_{\infty}\prth{t}}{L^{\infty}\prth{\mathbb{R}^{N}}}
\leq
k
&\hspace{\gap}\Prth{\vertii{u_{0}}{L^{1}\prth{\mathbb{R}^{N}}}+\vertii{v_{0}}{L^{1}\prth{\mathbb{R}^{N}}}}
e^{-t \frac{\Prth{\sqrt{\mu}+\sqrt{\nu}\hspace{0.3mm}}^{2}}{2}}, \\[2mm]
\vertii{\v\prth{t}-\v_{\!\infty}\prth{t}}{L^{\infty}\prth{\mathbb{R}^{N}}}
\leq
k'
&\hspace{\gap}\Prth{\vertii{u_{0}}{L^{1}\prth{\mathbb{R}^{N}}}+\vertii{v_{0}}{L^{1}\prth{\mathbb{R}^{N}}}}
e^{-t \frac{\Prth{\sqrt{\mu}+\sqrt{\nu}\hspace{0.3mm}}^{2}}{2}},
\end{array}$
}
we can write the following for $t>1$:
\renewcommand{\gap}{-3mm}
\begin{equation}\label{CONTROL_quick_glance_ue_ve_decay_rate_diff}
\begin{array}{ll}
	\vertii{\u_{e}\prth{t}}{L^{\infty}\prth{\mathbb{R}^{N}}}
	\leq
	\widetilde{\ell}_{e}
	&\hspace{\gap}\left.\Prth{\vertii{u_{0}}{L^{1}\prth{\mathbb{R}^{N}}}+\vertii{v_{0}}{L^{1}\prth{\mathbb{R}^{N}}}}\right/t^{N/2}, \\[2.5mm] 
	\vertii{\v_{\!e}\prth{t}}{L^{\infty}\prth{\mathbb{R}^{N}}}
	\leq
	\widetilde{\ell'}_{\!\!e}
	&\hspace{\gap}\left.\Prth{\vertii{u_{0}}{L^{1}\prth{\mathbb{R}^{N}}}+\vertii{v_{0}}{L^{1}\prth{\mathbb{R}^{N}}}}\right/t^{N/2},
\end{array}
\end{equation}
where $\widetilde{\ell}_{e}$ and $\widetilde{\ell'}_{\!\!e}$ are positive constants depending on $\mu$, $\nu$ and respectively $k$ and $k'$
--- hence
$N,c,d,\mu,\nu$.

We then shift to the consistent part of this proof that concerns the estimations of $\u_{\infty}$ and $\v_{\!\infty}$. Based on the Taylor expansion of $L$ near the origin, specifically,
\begin{equation*}
L\prth{\xi} = - \frac{c\nu+d\mu}{\mu+\nu}\verti{\xi}^{2} + o\Prth{\verti{\xi}^{2}},
\qquad
\text{as }\verti{\xi}\to 0,
\end{equation*}
we can chose a small enough positive constant $a$ (depending on $L$, hence $c,d,\mu,\nu$) so that
\begin{equation}\label{CONTROL_L_low_freq}
L\prth{\xi} \leq - \frac{c\nu+d\mu}{2\prth{\mu+\nu}}\verti{\xi}^{2},
\qquad
\text{as soon as }\verti{\xi}\leq a.
\end{equation}
Basic calculus may be employed to show that $L$ is radially decreasing.
This enables to find a positive constant $\eta$ (depending on $L$ and $a$, and consequently on $c,d,\mu,\nu$)
that guarantees
\begin{equation}\label{CONTROL_L_high_freq}
L\prth{\xi} \leq - \eta,
\qquad
\text{as soon as }\verti{\xi}\geq a.
\end{equation}
In the same manner as we demonstrated the vanishing of $\prth{\u_{e},\v_{\!e}}$, we use the Hausdorff-Young inequalities
\reff{INEQUALITY_Hausdorff_Young}{%
$
\begin{array}{l}
	\vertii{f}{L^{\infty}\prth{\mathbb{R}^{N}}} \leq
\prth{2\pi}^{-N}
\vertii{\hat{f}}{L^{1}\prth{\mathbb{R}^{N}}}, \\[2mm]
	\vertii{\hat{f}}{L^{\infty}\prth{\mathbb{R}^{N}}} \leq
\vertii{f}{L^{1}\prth{\mathbb{R}^{N}}}. 
\end{array}
$
}
to control
$\vertii{\u_{\infty}\prth{t}}{L^{\infty}\prth{\mathbb{R}^{N}}}$ and
$\vertii{\v_{\!\infty}\prth{t}}{L^{\infty}\prth{\mathbb{R}^{N}}}$
through estimates on
$\vertii{\,\hatt{\u}_{\infty}\prth{t}}{L^{1}\prth{\mathbb{R}^{N}}}$ and
$\vertii{\,\hatt{\v}_{\!\infty}\prth{t}}{L^{1}\prth{\mathbb{R}^{N}}}$.
Starting with $\u_{\infty}$, we have
\renewcommand{\gap}{-25mm}
\renewcommand{\gapp}{-22mm}
\begin{align}
 	\vertii{\,\hat{\u}_{\infty}\prth{t}}{L^{1}\prth{\mathbb{R}^{N}}} 
	&\,\leq &\hspace{\gap}\underbrace{\int_{\verti{\xi}\leq a}^{}\verti{\,\hat{\u}_{\infty}\prth{t,\xi}}d\xi}_{} &\,\quad +
	      &\hspace{\gapp}\underbrace{\int_{\verti{\xi}\geq a}^{}\verti{\,\hat{\u}_{\infty}\prth{t,\xi}}d\xi}_{}\label{GO_BACK}\\[-2.5mm]
 	&\,= :  &\hspace{\gap}\vertii{\uLow\prth{t}}{L^{1}\prth{\mathbb{R}^{N}}} &\,\quad +
 	      &\hspace{\gapp}\vertii{\uHigh\prth{t}}{L^{1}\prth{\mathbb{R}^{N}}}.\nonumber
\end{align}
For the high frequencies, we express $\hatt{\u}_{\infty}$ from\renewcommand{\gap}{3mm}\renewcommand{\gapp}{2mm}
\reff{EQ_persistent_part_Fourier_side}{%
$\begin{pmatrix} \,\hat{\u}_{\infty}\prth{t,\xi} \\[\gapp] \,\hat{\v}_{\!\infty}\prth{t,\xi} \end{pmatrix}
= : 
\frac{e^{t\lambda_{+}}}{2}
\begin{pmatrix}
1 - \frac{r}{\sqrt{s}}&
\frac{\nu}{\sqrt{s}}\\[\gap]
\frac{\mu}{\sqrt{s}} &
1 + \frac{r}{\sqrt{s}}
\end{pmatrix}
\begin{pmatrix} \,\hat{u}_{0}\prth{\xi} \\[\gapp] \,\hat{v}_{\skp\skp\skp 0}\prth{\xi} \end{pmatrix}.$
}
with $\lambda_{+}=L$. This yields
\begin{equation*}
\vertii{\uHigh\prth{t}}{L^{1}\prth{\mathbb{R}^{N}}} =
\frac{1}{2} \int_{\verti{\xi}\geq a}^{}\verti{\Prth{1-\frac{r}{\sqrt{s}}}\hat{u}_{0}\prth{\xi} + \frac{\nu}{\sqrt{s}}\,\hat{v}_{\skp\skp\skp 0}\prth{\xi}}e^{tL\prth{\xi}}d\xi.
\end{equation*}
Next, using controls on $L$
\reff{CONTROL_L_high_freq}{%
$L\prth{\xi} \leq - \eta,
\qquad
\text{as soon as }\verti{\xi}\geq a.$
}
and on $1-r/\sqrt{s}$ and ${\nu}/{\sqrt{s}}$
\reff{CONTROL_elts_evanescent_part}{%
$\verti{1\pm\frac{r}{\sqrt{s}}} \leq 2,
\qquad
\frac{\nu}{\sqrt{s}} \leq \sqrt{\frac{\nu}{\mu}},
\qquad
\frac{\mu}{\sqrt{s}} \leq \sqrt{\frac{\mu}{\nu}}.$
},
we eventually get
\begin{equation}\label{CONTROL_hat_u_infty_high}
\vertii{\uHigh\prth{t}}{L^{1}\prth{\mathbb{R}^{N}}} \leq
e^{-t\eta}\Prth{\vertii{\,\hat{u}_{0}}{L^{1}\prth{\mathbb{R}^{N}}}+\sqrt{\frac{\nu}{4\mu}}\vertii{\,\hat{v}_{\skp\skp\skp 0}}{L^{1}\prth{\mathbb{R}^{N}}}}
\end{equation}
which collapses exponentially fast. For the low frequencies, we similarly have
$$
\vertii{\uLow\prth{t}}{L^{1}\prth{\mathbb{R}^{N}}} =
\frac{1}{2} \int_{\verti{\xi}\leq a}^{}\verti{\Prth{1-\frac{r}{\sqrt{s}}}\hat{u}_{0}\prth{\xi} + \frac{\nu}{\sqrt{s}}\,\hat{v}_{\skp\skp\skp 0}\prth{\xi}}e^{tL\prth{\xi}}d\xi.
$$
Using controls on $L$
\reff{CONTROL_L_low_freq}{%
$L\prth{\xi} \leq - \frac{c\nu+d\mu}{2\prth{\mu+\nu}}\verti{\xi}^{2},
\qquad
\text{as soon as }\verti{\xi}\leq a.$
},
on $1-r/\sqrt{s}$ and ${\nu}/{\sqrt{s}}$
\reff{CONTROL_elts_evanescent_part}{%
$\verti{1\pm\frac{r}{\sqrt{s}}} \leq 2,
\qquad
\frac{\nu}{\sqrt{s}} \leq \sqrt{\frac{\nu}{\mu}},
\qquad
\frac{\mu}{\sqrt{s}} \leq \sqrt{\frac{\mu}{\nu}}.$
},
and
on $\hatt{u}_{0}$ and $\hatt{v}_{\skp\skp\skp 0}$
\reff{CONTROL_hat_u0_v0}{%
$\Big\vert\,\hat{u}_{0}\prth{\xi}\Big\vert\leq
\vertii{u_{0}}{L^{1}\prth{\mathbb{R}^{N}}},
\qquad
\Big\vert\,\hat{v}_{\skp\skp\skp 0}\prth{\xi}\Big\vert\leq 
\vertii{v_{0}}{L^{1}\prth{\mathbb{R}^{N}}}.$
},
we find:
\begin{align}
 	\vertii{\uLow\prth{t}}{L^{1}\prth{\mathbb{R}^{N}}} 
 	&\leq
\Prth{\vertii{u_{0}}{L^{1}\prth{\mathbb{R}^{N}}}+\sqrt{\frac{\nu}{4\mu}}\vertii{v_{0}}{L^{1}\prth{\mathbb{R}^{N}}}}
\int_{\verti{\xi}\leq a}^{}e^{-t\frac{c\nu+d\mu}{2\prth{\mu+\nu}}\verti{\xi}^{2}}d\xi \nonumber\\
	&\leq \Prth{\frac{2\pi\prth{\mu+\nu}}{c\nu+d\mu}}^{N/2}
	\left.\Prth{\vertii{u_{0}}{L^{1}\prth{\mathbb{R}^{N}}}+\sqrt{\frac{\nu}{4\mu}}\vertii{v_{0}}{L^{1}\prth{\mathbb{R}^{N}}}}\right/{t^{N/2}}.\label{CONTROL_hat_u_infty_low}
\end{align}
Finally, we combine
\reff{CONTROL_hat_u_infty_high}{%
$\vertii{\uHigh\prth{t}}{L^{1}\prth{\mathbb{R}^{N}}} \leq
e^{-t\eta}\Prth{\vertii{\,\hat{u}_{0}}{L^{1}\prth{\mathbb{R}^{N}}}+\sqrt{\frac{\nu}{4\mu}}\vertii{\,\hat{v}_{\skp\skp\skp 0}}{L^{1}\prth{\mathbb{R}^{N}}}}.$
}
and
\reff{CONTROL_hat_u_infty_low}{%
$\vertii{\uLow\prth{t}}{L^{1}\prth{\mathbb{R}^{N}}}
\leq \Prth{\frac{2\pi\prth{\mu+\nu}}{c\nu+d\mu}}^{N/2}
\left.\Prth{\vertii{u_{0}}{L^{1}\prth{\mathbb{R}^{N}}}+\sqrt{\frac{\nu}{4\mu}}\vertii{v_{0}}{L^{1}\prth{\mathbb{R}^{N}}}}\right/{t^{N/2}}.$
}
and use Hausdorff-Young inequalities
\reff{INEQUALITY_Hausdorff_Young}{%
$
\begin{array}{l}
	\vertii{f}{L^{\infty}\prth{\mathbb{R}^{N}}} \leq
\prth{2\pi}^{-N}
\vertii{\hat{f}}{L^{1}\prth{\mathbb{R}^{N}}}, \\[2mm]
	\vertii{\hat{f}}{L^{\infty}\prth{\mathbb{R}^{N}}} \leq
\vertii{f}{L^{1}\prth{\mathbb{R}^{N}}}. 
\end{array}
$
}
to control $\u_{\infty}\prth{t}$ in $L^{\infty}\prth{\mathbb{R}^{N}}$.
By returning to\renewcommand{\gap}{-25mm}\renewcommand{\gapp}{-22mm}
\reff{GO_BACK}{%
$\vertii{\,\hat{\u}_{\infty}\prth{t}}{L^{1}\prth{\mathbb{R}^{N}}} \leq
\vertii{\uLow\prth{t}}{L^{1}\prth{\mathbb{R}^{N}}} +
\vertii{\uHigh\prth{t}}{L^{1}\prth{\mathbb{R}^{N}}}.$
},
we can follow the same approach to estimate $\v_{\!\infty}\prth{t}$ as well.
As a consequence, we can identify two positive constants
$\widetilde{\ell}_{\infty}$ and
$\widetilde{\ell'}_{\!\!\infty}$,
which depend on
$N,c,d,\mu,\nu$,
and satisfy the following inequalities for $t > 1$:
\renewcommand{\gap}{-3.5mm}
\begin{equation}\label{CONTROL_u_infty_v_infty_decay_rate_diff}
\!\!\!\begin{array}{ll}
	\vertii{\u_{\infty}\prth{t}}{L^{\infty}\prth{\mathbb{R}^{N}}}
	\leq
	\widetilde{\ell}_{\infty}
	&\hspace{\gap}\Big(
	\vertii{u_{0}}{L^{1}\prth{\mathbb{R}^{N}}}+
	\vertii{v_{0}}{L^{1}\prth{\mathbb{R}^{N}}}+
	\vertii{\?\?\?\?\?\hat{u}_{0}}{L^{1}\prth{\mathbb{R}^{N}}}+
	\vertii{\?\?\?\?\hat{v}_{\skp\skp\skp 0}}{L^{1}\prth{\mathbb{R}^{N}}}\left.\Big)\right/t^{N/2},\\[2.5mm] 
	\vertii{\v_{\!\infty}\prth{t}}{L^{\infty}\prth{\mathbb{R}^{N}}}
	\leq
	\widetilde{\ell'}_{\!\!\infty}
	&\hspace{\gap}\Big(
	\vertii{u_{0}}{L^{1}\prth{\mathbb{R}^{N}}}+
	\vertii{v_{0}}{L^{1}\prth{\mathbb{R}^{N}}}+
	\vertii{\?\?\?\?\?\hat{u}_{0}}{L^{1}\prth{\mathbb{R}^{N}}}+
	\vertii{\?\?\?\?\hat{v}_{\skp\skp\skp 0}}{L^{1}\prth{\mathbb{R}^{N}}}\left.\Big)\right/t^{N/2}.
\end{array}
\end{equation}
To complete the control of $\prth{\u,\v}$ for $t>1$, we combine the estimations on
$\prth{\u_{e},\v_{\!e}}$
\reff{CONTROL_quick_glance_ue_ve_decay_rate_diff}{%
$\begin{array}{ll}
	\vertii{\u_{e}\prth{t}}{L^{\infty}\prth{\mathbb{R}^{N}}}
	\leq
	\widetilde{\ell}_{e}
	&\hspace{\gap}\left.\Prth{\vertii{u_{0}}{L^{1}\prth{\mathbb{R}^{N}}}+\vertii{v_{0}}{L^{1}\prth{\mathbb{R}^{N}}}}\right/t^{N/2}, \\[2.5mm] 
	\vertii{\v_{\!e}\prth{t}}{L^{\infty}\prth{\mathbb{R}^{N}}}
	\leq
	\widetilde{\ell'}_{\!\!e}
	&\hspace{\gap}\left.\Prth{\vertii{u_{0}}{L^{1}\prth{\mathbb{R}^{N}}}+\vertii{v_{0}}{L^{1}\prth{\mathbb{R}^{N}}}}\right/t^{N/2}.
\end{array}$
}
and
$\prth{\u_{\infty},\v_{\!\infty}}$\renewcommand{\gap}{-3.5mm}
\reff{CONTROL_u_infty_v_infty_decay_rate_diff}{%
$
\begin{array}{ll}
	\vertii{\u_{\infty}\prth{t}}{L^{\infty}\prth{\mathbb{R}^{N}}}
	\leq
	\widetilde{\ell}_{\infty}
	&\hspace{\gap}\Big(
	\vertii{u_{0}}{L^{1}\prth{\mathbb{R}^{N}}}+
	\vertii{v_{0}}{L^{1}\prth{\mathbb{R}^{N}}}+
	\vertii{\?\?\?\?\?\hat{u}_{0}}{L^{1}\prth{\mathbb{R}^{N}}}+
	\vertii{\?\?\?\?\hat{v}_{\skp\skp\skp 0}}{L^{1}\prth{\mathbb{R}^{N}}}\left.\Big)\right/t^{N/2},\\[2.5mm] 
	\vertii{\v_{\!\infty}\prth{t}}{L^{\infty}\prth{\mathbb{R}^{N}}}
	\leq
	\widetilde{\ell'}_{\!\!\infty}
	&\hspace{\gap}\Big(
	\vertii{u_{0}}{L^{1}\prth{\mathbb{R}^{N}}}+
	\vertii{v_{0}}{L^{1}\prth{\mathbb{R}^{N}}}+
	\vertii{\?\?\?\?\?\hat{u}_{0}}{L^{1}\prth{\mathbb{R}^{N}}}+
	\vertii{\?\?\?\?\hat{v}_{\skp\skp\skp 0}}{L^{1}\prth{\mathbb{R}^{N}}}\left.\Big)\right/t^{N/2}.
\end{array}$
}
to obtain\renewcommand{\gap}{-3.5mm}
\reff{CONTROL_heristic_uv_decay_rate_diff}{%
$\begin{array}{ll}
	\vertii{\u\prth{t}}{L^{\infty}\prth{\mathbb{R}^{N}}}
	\leq
	\widetilde{\ell}
	&\hspace{\gap}\Big(
	\vertii{u_{0}}{L^{1}\prth{\mathbb{R}^{N}}}+
	\vertii{v_{0}}{L^{1}\prth{\mathbb{R}^{N}}}+
	\vertii{\?\?\?\?\?\hat{u}_{0}}{L^{1}\prth{\mathbb{R}^{N}}}+
	\vertii{\?\?\?\?\hat{v}_{\skp\skp\skp 0}}{L^{1}\prth{\mathbb{R}^{N}}}\left.\Big)\right/t^{N/2}, \\[2.5mm] 
	\vertii{\v\prth{t}}{L^{\infty}\prth{\mathbb{R}^{N}}}
	\leq
	\widetilde{\ell'}
	&\hspace{\gap}\Big(
	\vertii{u_{0}}{L^{1}\prth{\mathbb{R}^{N}}}+
	\vertii{v_{0}}{L^{1}\prth{\mathbb{R}^{N}}}+
	\vertii{\?\?\?\?\?\hat{u}_{0}}{L^{1}\prth{\mathbb{R}^{N}}}+
	\vertii{\?\?\?\?\hat{v}_{\skp\skp\skp 0}}{L^{1}\prth{\mathbb{R}^{N}}}\left.\Big)\right/t^{N/2}.
\end{array}$
}
with
$$
\widetilde{\ell} : =
\widetilde{\ell}_{e}+
\widetilde{\ell}_{\infty}
\qquad 
\text{and}
\qquad 
\widetilde{\ell'} : =
\widetilde{\ell'}_{\!\!e}+
\widetilde{\ell'}_{\!\!\infty}.
$$

Now, for $t\leq 1$, the comparison principle ensures that $\prth{\u,\v}$ stays below the solution
$\prth{\bara{\u},\bara{\v}}$
to
\reff{SYS_heat_exchanger_diff}{%
$\left\lbrace\begin{array}{l}
	\partial_{t}\u = c\Delta \u - \mu \u + \nu\v,\\
	\partial_{t}\v = d\Delta \v + \mu \u - \nu\v.
\end{array}\right.$
}
with initial condition
$\prth{\bara{u}_{\skp\skp 0},\bara{v}_{\!0}}\equiv \prth{\vertii{u_{0}}{L^{\infty}\prth{\mathbb{R}^{N}}},\vertii{v_{0}}{L^{\infty}\prth{\mathbb{R}^{N}}}}$.
Notably,
$\prth{\bara{\u},\bara{\v}\?\?\?}$
actually does not depend on the space variable and therefore solves the ODE system
\begin{equation*}
\renewcommand{\gap}{-3mm}
\left\lbrace
\begin{array}{llll}
	\partial_{t}\bara{\u} = - &\hspace{\gap}\mu\bara{\u} + \nu\bara{\v}, & \qquad & t>0,\\
	\partial_{t}\bara{\v} =   &\hspace{\gap}\mu\bara{\u} - \nu\bara{\v}, & \qquad & t>0.
\end{array}
\right.
\end{equation*}
Letting
$\bara{\sigmaSR} : = \bara{\u}+\bara{\v}$,
we clearly have
$\partial_{t}\bara{\sigmaSR} \equiv 0$,
so that,
\begin{align*}
\max\prth{\u , \v} \leq 
\max\prth{\bara{\u},\bara{\v}} \leq 
\bara{\sigmaSR} &\equiv
\vertii{u_{0}}{L^{\infty}\prth{\mathbb{R}^{N}}}+\vertii{v_{0}}{L^{\infty}\prth{\mathbb{R}^{N}}}\\[1mm]
&\leq
\prth{2\pi}^{-N}
\Big(\vertii{\?\?\?\?\?\hat{u}_{0}}{L^{1}\prth{\mathbb{R}^{N}}}+
\vertii{\?\?\?\?\hat{v}_{\skp\skp\skp 0}}{L^{1}\prth{\mathbb{R}^{N}}}\Big),
\end{align*}
where we used the Hausdorff-Young inequalities
\reff{INEQUALITY_Hausdorff_Young}{%
$
\begin{array}{l}
	\vertii{f}{L^{\infty}\prth{\mathbb{R}^{N}}} \leq
\prth{2\pi}^{-N}
\vertii{\hat{f}}{L^{1}\prth{\mathbb{R}^{N}}}, \\[2mm]
	\vertii{\hat{f}}{L^{\infty}\prth{\mathbb{R}^{N}}} \leq
\vertii{f}{L^{1}\prth{\mathbb{R}^{N}}}. 
\end{array}
$
}
to obtain the last line.
As a consequence, we have,
\renewcommand{\gap}{-3.5mm}
\begin{equation}\label{CONTROL_t_leq_one_uv_decay_rate_diff}
\!\!\begin{array}{ll}
	\vertii{\u\prth{t}}{L^{\infty}\prth{\mathbb{R}^{N}}}
	\leq
	\prth{2\pi}^{-N}
	&\hspace{\gap}\Big(
	\vertii{u_{0}}{L^{1}\prth{\mathbb{R}^{N}}}+
	\vertii{v_{0}}{L^{1}\prth{\mathbb{R}^{N}}}+
	\vertii{\?\?\?\?\?\hat{u}_{0}}{L^{1}\prth{\mathbb{R}^{N}}}+
	\vertii{\?\?\?\?\hat{v}_{\skp\skp\skp 0}}{L^{1}\prth{\mathbb{R}^{N}}}\Big),\\[2.5mm] 
	\vertii{\v\prth{t}}{L^{\infty}\prth{\mathbb{R}^{N}}}
	\leq
	\prth{2\pi}^{-N}
	&\hspace{\gap}\Big(
	\vertii{u_{0}}{L^{1}\prth{\mathbb{R}^{N}}}+
	\vertii{v_{0}}{L^{1}\prth{\mathbb{R}^{N}}}+
	\vertii{\?\?\?\?\?\hat{u}_{0}}{L^{1}\prth{\mathbb{R}^{N}}}+
	\vertii{\?\?\?\?\hat{v}_{\skp\skp\skp 0}}{L^{1}\prth{\mathbb{R}^{N}}}\Big),
\end{array}
\end{equation}
for all $t\leq 1$.

To complete the proof, we need to find suitable values for $\ell$ and $\ell'$ such that\renewcommand{\gap}{-3.5mm}
\reff{CONTROL_uv_decay_rate_diff}{%
$\begin{array}{ll}
	\vertii{\u\prth{t}}{L^{\infty}\prth{\mathbb{R}^{N}}}
	\leq
	\ell
	&\hspace{\gap}\Big(
	\vertii{u_{0}}{L^{1}\prth{\mathbb{R}^{N}}}\skp+\skp
	\vertii{v_{0}}{L^{1}\prth{\mathbb{R}^{N}}}\skp+\skp
	\vertii{\?\?\?\?\?\hat{u}_{0}}{L^{1}\prth{\mathbb{R}^{N}}}\skp+\skp
	\vertii{\?\?\?\?\hat{v}_{\skp\skp\skp 0}}{L^{1}\prth{\mathbb{R}^{N}}}\left.\Big)\right/\prth{1+t}^{N/2},\\[2.5mm] 
	\vertii{\v\prth{t}}{L^{\infty}\prth{\mathbb{R}^{N}}}
	\leq
	\ell'
	&\hspace{\gap}\Big(
	\vertii{u_{0}}{L^{1}\prth{\mathbb{R}^{N}}}\skp+\skp
	\vertii{v_{0}}{L^{1}\prth{\mathbb{R}^{N}}}\skp+\skp
	\vertii{\?\?\?\?\?\hat{u}_{0}}{L^{1}\prth{\mathbb{R}^{N}}}\skp+\skp
	\vertii{\?\?\?\?\hat{v}_{\skp\skp\skp 0}}{L^{1}\prth{\mathbb{R}^{N}}}\left.\Big)\right/\prth{1+t}^{N/2},\\[2.5mm] 
	\text{for all }t>0.&~
\end{array}$
}
satisfies both estimates\renewcommand{\gap}{-3.5mm}
\reff{CONTROL_t_leq_one_uv_decay_rate_diff}{%
$\begin{array}{ll}
	\vertii{\u\prth{t}}{L^{\infty}\prth{\mathbb{R}^{N}}}
	\leq
	\prth{2\pi}^{-N}
	&\hspace{\gap}\Big(
	\vertii{u_{0}}{L^{1}\prth{\mathbb{R}^{N}}}+
	\vertii{v_{0}}{L^{1}\prth{\mathbb{R}^{N}}}+
	\vertii{\?\?\?\?\?\hat{u}_{0}}{L^{1}\prth{\mathbb{R}^{N}}}+
	\vertii{\?\?\?\?\hat{v}_{\skp\skp\skp 0}}{L^{1}\prth{\mathbb{R}^{N}}}\Big),\\[2.5mm] 
	\vertii{\v\prth{t}}{L^{\infty}\prth{\mathbb{R}^{N}}}
	\leq
	\prth{2\pi}^{-N}
	&\hspace{\gap}\Big(
	\vertii{u_{0}}{L^{1}\prth{\mathbb{R}^{N}}}+
	\vertii{v_{0}}{L^{1}\prth{\mathbb{R}^{N}}}+
	\vertii{\?\?\?\?\?\hat{u}_{0}}{L^{1}\prth{\mathbb{R}^{N}}}+
	\vertii{\?\?\?\?\hat{v}_{\skp\skp\skp 0}}{L^{1}\prth{\mathbb{R}^{N}}}\Big).
\end{array}$
}
(for $t\leq 1$) and\renewcommand{\gap}{-3.5mm}
\reff{CONTROL_heristic_uv_decay_rate_diff}{%
$\begin{array}{ll}
	\vertii{\u\prth{t}}{L^{\infty}\prth{\mathbb{R}^{N}}}
	\leq
	\widetilde{\ell}
	&\hspace{\gap}\Big(
	\vertii{u_{0}}{L^{1}\prth{\mathbb{R}^{N}}}+
	\vertii{v_{0}}{L^{1}\prth{\mathbb{R}^{N}}}+
	\vertii{\?\?\?\?\?\hat{u}_{0}}{L^{1}\prth{\mathbb{R}^{N}}}+
	\vertii{\?\?\?\?\hat{v}_{\skp\skp\skp 0}}{L^{1}\prth{\mathbb{R}^{N}}}\left.\Big)\right/t^{N/2}, \\[2.5mm] 
	\vertii{\v\prth{t}}{L^{\infty}\prth{\mathbb{R}^{N}}}
	\leq
	\widetilde{\ell'}
	&\hspace{\gap}\Big(
	\vertii{u_{0}}{L^{1}\prth{\mathbb{R}^{N}}}+
	\vertii{v_{0}}{L^{1}\prth{\mathbb{R}^{N}}}+
	\vertii{\?\?\?\?\?\hat{u}_{0}}{L^{1}\prth{\mathbb{R}^{N}}}+
	\vertii{\?\?\?\?\hat{v}_{\skp\skp\skp 0}}{L^{1}\prth{\mathbb{R}^{N}}}\left.\Big)\right/t^{N/2}.
\end{array}$
}
(for $t>1$).
It can be shown that choosing
$$
\ell : =
2^{N/2}\max\Prth{\prth{2\pi}^{-N} , \widetilde{\ell}\,\?\?\?}
\qquad 
\text{and}
\qquad 
\ell' : =
2^{N/2}\max\Prth{\prth{2\pi}^{-N} , \widetilde{\ell'}\?\?\?}
$$
is the optimal solution, so that these values inherit the parameters dependency specified in
Corollary \ref{CORO_decay_linear_Heat_exchanger}.
\end{proof}

\section{Possible global existence}\label{S4_possible_global_existence}

In this section we show the possible existence of global solutions for problem
\reff{SYS_heat_exchanger}{%
\renewcommand{\gap}{2.155mm}
$
\left\lbrace \begin{array}{l}
	\partial_{t}u = c\Delta u - \mu u + \nu v + \hspace{\gap}u^{1+\p},\\
	\partial_{t}v = d\Delta v + \mu u - \nu v + \kappa v^{1+\q},\\
	\prth{u,v}|_{t=0} = \prth{u_{0},v_{0}}.
\end{array} \right.
$
}
when $N/2$ is larger than both $1/\p$ and $\kappa/\q$, as stated in Theorem \ref{TH_possible_global_existence}.

The proof involves constructing a global super-solution $\prth{\bara{u},\bara{v}}$ for problem
\reff{SYS_heat_exchanger}{%
\renewcommand{\gap}{2.155mm}
$
\left\lbrace \begin{array}{l}
	\partial_{t}u = c\Delta u - \mu u + \nu v + \hspace{\gap}u^{1+\p},\\
	\partial_{t}v = d\Delta v + \mu u - \nu v + \kappa v^{1+\q},\\
	\prth{u,v}|_{t=0} = \prth{u_{0},v_{0}}.
\end{array} \right.
$
}
that drives $u$ and $v$ to $0$.
More precisely, we try
\begin{equation}\label{DEF_guess_sur_sol}
\bara{u}\prth{t,x} : = F\prth{t}\times\u\prth{t,x}
\qquad 
\text{and}
\qquad 
\bara{v}\prth{t,x} : = F\prth{t}\times\v\prth{t,x},
\end{equation}
where $\prth{\u,\v}$ is the solution to the pure diffusive Heat exchanger 
\reff{SYS_heat_exchanger_diff}{%
$\left\lbrace\begin{array}{l}
	\partial_{t}\u = c\Delta \u - \mu \u + \nu\v,\\
	\partial_{t}\v = d\Delta \v + \mu \u - \nu\v.
\end{array}\right.$
}
with initial condition $\prth{u_{0},v_{0}}$, and $F$ is a positive, bounded and continuously differentiable function
to be determined.
We then search for sufficient conditions on $F$ that guarantee $\prth{\bara{u},\bara{v}}$ is a global super-solution.
This leads to
\reff{CONTROL_pour_avoir_existence_globale}{%
$m : = \vertii{u_{0}}{L^{1}\prth{\mathbb{R}^{N}}} + 
\vertii{v_{0}}{L^{1}\prth{\mathbb{R}^{N}}} + 
\vertii{\?\?\?\?\?\hat{u}_{0}}{L^{1}\prth{\mathbb{R}^{N}}}\skp+\skp
\vertii{\?\?\?\?\hat{v}_{\skp\skp\skp 0}}{L^{1}\prth{\mathbb{R}^{N}}}
< m_{0}.$
}
which requires the datum $\prth{u_{0},v_{0}}$ to be small in some sense.
Finally, the upper control $\prth{u,v}\leq \sup\prth{F}\times\prth{\u,\v}$ retrieves
\reff{CONTROL_solution_globale}{%
$\begin{array}{l}
	\vertii{u\prth{t}}{L^{\infty}\prth{\mathbb{R}^{N}}}
\leq
\frac{M}{\prth{1+t}^{N/2}},
	\qquad
	\vertii{v\prth{t}}{L^{\infty}\prth{\mathbb{R}^{N}}}
\leq
\frac{M'}{\prth{1+t}^{N/2}},\\[2mm]
	\text{for all }t>0.
\end{array}$
}.

Observe that our primary focus is on problem
\reff{SYS_heat_exchanger}{%
\renewcommand{\gap}{2.155mm}
$
\left\lbrace \begin{array}{l}
	\partial_{t}u = c\Delta u - \mu u + \nu v + u^{1+\p},\\
	\partial_{t}v = d\Delta v + \mu u - \nu v + v^{1+\q},\\
	\prth{u,v}|_{t=0} = \prth{u_{0},v_{0}}.
\end{array} \right.
$
}$|_{\kappa=1}$
to prove Theorem \ref{TH_possible_global_existence}.
However, we also investigate the simpler case
\reff{SYS_heat_exchanger}{%
\renewcommand{\gap}{2.155mm}
$
\left\lbrace \begin{array}{l}
	\partial_{t}u = c\Delta u - \mu u + \nu v + u^{1+\p},\\
	\partial_{t}v = d\Delta v + \mu u - \nu v,\\
	\prth{u,v}|_{t=0} = \prth{u_{0},v_{0}}.
\end{array} \right.
$
}$|_{\kappa=0}$
to obtain more accurate values for the constants $m_{0}$, $M$ and $M'$ when $\kappa=0$.

\bigskip

\begin{proof}[Proof of Theorem \ref{TH_possible_global_existence}]
Consider $\prth{\bara{u},\bara{v}}$ as defined in
\reff{DEF_guess_sur_sol}{%
$\bara{u}\prth{t,x} : = F\prth{t}\times\u\prth{t,x}
\qquad 
\text{and}
\qquad 
\bara{v}\prth{t,x} : = F\prth{t}\times\v\prth{t,x}.$
}.
First, we set $F\prth{0}=1$ ensuring that
$\prth{\bara{u},\bara{v}}$ and $\prth{u,v}$ have the same initial data.
Next, we require the following expressions to be positive for all $t>0$ and $x\in \mathbb{R}^{N}$:
$$
\partial_{t}\bara{u} - c\Delta \bara{u} + \mu \bara{u} - \nu \bara{v} - \bara{u}^{1+\p}
\qquad 
\text{and}
\qquad 
\partial_{t}\bara{v} - d\Delta \bara{v} - \mu \bara{u} + \nu \bara{v} - \kappa \bara{v}^{1+\q}.
$$
This leads to
\begin{equation}\label{a1}
F'\geq F^{1+\p} \times \u^{\p}
\qquad 
\text{and}
\qquad 
F'\geq \kappa \times F^{1+\q} \times \v^{\?\?\?\q}.
\end{equation}
Note that $\u$ and $\v$ can be replaced in
\reff{a1}{%
$F'\geq F^{1+\p} \times \u^{\p}
\qquad 
\text{and}
\qquad 
F'\geq \kappa \times F^{1+\q} \times \v^{\?\?\?\q}.$
}
by their $L^{\infty}\prth{\mathbb{R}^{N}}$-norms and any control from above of them.
Therefore, with Corollary \ref{CORO_decay_linear_Heat_exchanger} providing uniform controls on $\u$ and $\v$, we can say that satisfying the following inequalities is sufficient to recover
\reff{a1}{%
$F'\geq F^{1+\p} \times \u^{\p}
\qquad 
\text{and}
\qquad 
F'\geq \kappa \times F^{1+\q} \times \v^{\?\?\?\q}.$
}:
\begin{equation}\label{a2}
F'\geq F^{1+\p} \times \Prth{\frac{\ell m}{\prth{1+t}^{N/2}}}^{\p}
\qquad 
\text{and}
\qquad 
F'\geq \kappa \times F^{1+\q} \times \Prth{\frac{\ell' m}{\prth{1+t}^{N/2}}}^{\q}.
\end{equation}
Moving forward, we split the discussion into two parts based on whether $\kappa$ is $0$ or $1$.

\bigskip

\noindent\hspace{-0.5mm}\textbullet\,\textit{The case $\kappa=0$.}
This case is the simplest of the two since it is sufficient to ask
\begin{equation}\label{a3}
F' = F^{1+\p} \times \frac{\prth{\ell m}^{\p}}{\prth{1+t}^{N\p/2}}
\end{equation}
if we require
\reff{a2}{%
$F'\geq F^{1+\p} \times \Prth{\frac{\ell m}{\prth{1+t}^{N/2}}}^{\p}
\qquad 
\text{and}
\qquad 
F'\geq \kappa \times F^{1+\q} \times \Prth{\frac{\ell' m}{\prth{1+t}^{N/2}}}^{\q}$
}.
By solving the ODE
\reff{a3}{%
$F' = F^{1+\p} \times \frac{\prth{\ell m}^{\p}}{\prth{1+t}^{N\p/2}}.$
} with $F\prth{0}=1$,
we obtain
\begin{equation}\label{a4}
F\prth{t} = \Bigg[\underbrace{1 - \frac{2\p\prth{\ell m}^{\p}}{N\p-2}\Prth{1-\frac{1}{\prth{1+t}^{\prth{N\p/2}-1}}}}_{=: G_{0}\prth{t}.}\Bigg]^{-1/\p}
\end{equation}
It remains to ensure that $F$ exists for all times which is permitted if and only if $G_{0}$ does not collide $0$.
To achieve this, we first need to assume that we are in the regime
\reff{PF_possible_extinction}{%
$\frac{N}{2}>
\left\lbrace \begin{array}{ll}
	\frac{1}{\p} &\quad \text{if }\kappa=0, \\[1mm]
	\max\Prth{\frac{1}{\p}, \frac{1}{\q}} &\quad \text{if }\kappa=1. \\
\end{array} \right.$}
to guarantee the vanishing of
$1/\prth{1+t}^{\prth{N\p/2}-1}$ in
\reff{a4}{%
$F\prth{t} = \Bigg[\underbrace{1 - \frac{2\p\prth{\ell m}^{\p}}{N\p-2}\Prth{1-\frac{1}{\prth{1+t}^{\prth{N\p/2}-1}}}}_{\text{Call that }G_{0}\prth{t}}\Bigg]^{-1/\p}$
}.
Then, because
\begin{equation}\label{a4inf}
\inf\limits_{t\geq 0} \crochets{G_{0}\prth{t}} = 1-\frac{2\p\prth{\ell m}^{\p}}{N\p-2},
\end{equation}
it suffices to chose
$$
m<m_{0} : = \Prth{\frac{N\p-2}{2\p\ell^{\p}}}^{1/\p}
$$
to make $G_{0}$ positive and so
$\prth{\bara{u},\bara{v}}$
global.

To eventually retrieve the controls in
\reff{CONTROL_solution_globale}{%
$\begin{array}{l}
	\vertii{u\prth{t}}{L^{\infty}\prth{\mathbb{R}^{N}}}
\leq
\frac{M}{\prth{1+t}^{N/2}},
	\qquad
	\vertii{v\prth{t}}{L^{\infty}\prth{\mathbb{R}^{N}}}
\leq
\frac{M'}{\prth{1+t}^{N/2}},\\[2mm]
	\text{for all }t>0.
\end{array}$
},
we combine
\reff{a4}{%
$F\prth{t} = \Bigg[\underbrace{1 - \frac{2\p\prth{\ell m}^{\p}}{N\p-2}\Prth{1-\frac{1}{\prth{1+t}^{\prth{N\p/2}-1}}}}_{\text{Call that }G_{0}\prth{t}}\Bigg]^{-1/\p}$
}-\reff{a4inf}{%
$\inf\limits_{t\geq 0} \crochets{G_{0}\prth{t}} = 1-\frac{2\p\prth{\ell m}^{\p}}{N\p-2}.$
}
and\renewcommand{\gap}{-3.5mm}
\reff{CONTROL_uv_decay_rate_diff}{%
$\begin{array}{ll}
	\vertii{\u\prth{t}}{L^{\infty}\prth{\mathbb{R}^{N}}}
	\leq
	\ell
	&\hspace{\gap}\Big(
	\vertii{u_{0}}{L^{1}\prth{\mathbb{R}^{N}}}\skp+\skp
	\vertii{v_{0}}{L^{1}\prth{\mathbb{R}^{N}}}\skp+\skp
	\vertii{\?\?\?\?\?\hat{u}_{0}}{L^{1}\prth{\mathbb{R}^{N}}}\skp+\skp
	\vertii{\?\?\?\?\hat{v}_{\skp\skp\skp 0}}{L^{1}\prth{\mathbb{R}^{N}}}\left.\Big)\right/\prth{1+t}^{N/2},\\[2.5mm] 
	\vertii{\v\prth{t}}{L^{\infty}\prth{\mathbb{R}^{N}}}
	\leq
	\ell'
	&\hspace{\gap}\Big(
	\vertii{u_{0}}{L^{1}\prth{\mathbb{R}^{N}}}\skp+\skp
	\vertii{v_{0}}{L^{1}\prth{\mathbb{R}^{N}}}\skp+\skp
	\vertii{\?\?\?\?\?\hat{u}_{0}}{L^{1}\prth{\mathbb{R}^{N}}}\skp+\skp
	\vertii{\?\?\?\?\hat{v}_{\skp\skp\skp 0}}{L^{1}\prth{\mathbb{R}^{N}}}\left.\Big)\right/\prth{1+t}^{N/2},\\[2.5mm] 
	\text{for all }t>0.&~
\end{array}$
}
in
Corollary \ref{CORO_decay_linear_Heat_exchanger}, which lead us to set
$$
M : = \Prth{\frac{\prth{N\p-2}\prth{\ell m}^{\p}}{N\p-2-2\p\prth{\ell m}^{\p}}}^{1/\p}
\qquad 
\text{and}
\qquad 
M' : = \Prth{\frac{\prth{N\p-2}\prth{\ell' m}^{\p}}{N\p-2-2\p\prth{\ell m}^{\p}}}^{1/\p},
$$
completing the proof in this case.

\bigskip

\noindent\hspace{-0.5mm}\textbullet\,\textit{The case $\kappa=1$.}
Similar to the previous case, we aim to set an ODE on $F$,
like
\reff{a3}{%
$F' = F^{1+\p} \times \frac{\prth{\ell m}^{\p}}{\prth{1+t}^{N\p/2}}.$
},
that would satisfy both ODIs in
\reff{a2}{%
$F'\geq F^{1+\p} \times \Prth{\frac{\ell m}{\prth{1+t}^{N/2}}}^{\p}
\qquad 
\text{and}
\qquad 
F'\geq \kappa \times F^{1+\q} \times \Prth{\frac{\ell' m}{\prth{1+t}^{N/2}}}^{\q}$
}.
With the loose assumption $m<1$, it is clear that
$$
\max\crochets{
\Prth{\frac{m}{\prth{1+t}^{N/2}}}^{\p},
\Prth{\frac{m}{\prth{1+t}^{N/2}}}^{\q}\,
}
\leq
\Prth{\frac{m}{\prth{1+t}^{N/2}}}^{\min\prth{\p,\q}}.
$$
Moreover,
\reff{a1}{%
$F'\geq F^{1+\p} \times \u^{\p}
\qquad 
\text{and}
\qquad 
F'\geq \kappa \times F^{1+\q} \times \v^{\?\?\?\q}.$
}
requires $F'$ to be non-negative, so $F\prth{t}\geq F\prth{0} = 1$ implies
$$
\max\Prth{
F^{1+\p},
F^{1+\q}
}
\leq
F^{1+\max\prth{\p,\q}}.
$$
As a result, it is sufficient to ask
\begin{equation}\label{a5}
F' = F^{1+\max\prth{\p,\q}}
\times \max\prth{\ell^{\p},\ell'^{\q}}
\times \Prth{\frac{m}{\prth{1+t}^{N/2}}}^{\min\prth{\p,\q}}
\end{equation}
for $F$ to satisfy
\reff{a2}{%
$F'\geq F^{1+\p} \times \Prth{\frac{\ell m}{\prth{1+t}^{N/2}}}^{\p}
\qquad 
\text{and}
\qquad 
F'\geq \kappa \times F^{1+\q} \times \Prth{\frac{\ell' m}{\prth{1+t}^{N/2}}}^{\q}$
}.
The remainder of the proof is nearly identical to the case $\kappa=0$. Solving
\reff{a5}{%
$F' = F^{1+\max\prth{\p,\q}}
\times \max\prth{\ell^{\p},\ell'^{\q}}
\times \Prth{\frac{m}{\prth{1+t}^{N/2}}}^{\min\prth{\p,\q}}.$
}
yields
\begin{equation}\label{a6}
\!\!\!\!\!\!\!
F\prth{t} \!=\! \Bigg[\underbrace{1 - \frac{2\max\prth{\p,\q}\max\prth{\ell^{\p},\ell'^{\q}}m^{\min\prth{\p,\q}}}{N\min\prth{\p,\q}-2}\Prth{1-\frac{1}{\prth{1+t}^{\prth{N\min\prth{\p,\q}/2}-1}}}}_{ = :G_{1}\prth{t}.}\Bigg]^{-1/\max\prth{\p,\q}}\!\!\!\!\!\!\!\!
\end{equation}
To ensure the global existence of $F$, we first need to be in the regime
\reff{PF_possible_extinction}{%
$\frac{N}{2}>
\left\lbrace \begin{array}{ll}
	\frac{1}{\p} &\quad \text{if }\kappa=0, \\[1mm]
	\max\Prth{\frac{1}{\p}, \frac{1}{\q}} &\quad \text{if }\kappa=1. \\
\end{array} \right.$}
to guarantee the vanishing of
$1/\prth{1+t}^{\prth{N\min\prth{\p,\q}/2}-1}$ in
\reff{a6}{%
$F\prth{t} \!=\! \Bigg[\underbrace{1 - \frac{2\max\prth{\p,\q}\max\prth{\ell^{\p},\ell'^{\q}}m^{\min\prth{\p,\q}}}{N\min\prth{\p,\q}-2}\Prth{1-\frac{1}{\prth{1+t}^{\prth{N\min\prth{\p,\q}/2}-1}}}}_{\text{Call that }G_{1}\prth{t}}\Bigg]^{-1/\max\prth{\p,\q}}$
}.
Then, because
\begin{equation}\label{a6inf}
\inf\limits_{t\geq 0} \crochets{G_{1}\prth{t}} = 1-\frac{2\max\prth{\p,\q}\max\prth{\ell^{\p},\ell'^{\q}}m^{\min\prth{\p,\q}}}{N\min\prth{\p,\q}-2},
\end{equation}
it suffices to chose
$$
m<m_{0} : = \min\crochets{1,\Prth{\frac{N\min\prth{\p,\q}-2}{2\max\prth{\p,\q}\max\prth{\ell^{\p},\ell'^{\q}}}}^{1/\min\prth{\p,\q}}}
$$
to make $G_{1}$ positive and thus
$\prth{\bara{u},\bara{v}}$
global.

It eventually remains to recover the controls in
\reff{CONTROL_solution_globale}{%
$\begin{array}{l}
	\vertii{u\prth{t}}{L^{\infty}\prth{\mathbb{R}^{N}}}
\leq
\frac{M}{\prth{1+t}^{N/2}},
	\qquad
	\vertii{v\prth{t}}{L^{\infty}\prth{\mathbb{R}^{N}}}
\leq
\frac{M'}{\prth{1+t}^{N/2}},\\[2mm]
	\text{for all }t>0.
\end{array}$
}.
To do this, we combine
\reff{a6}{%
$F\prth{t} \!=\! \Bigg[\underbrace{1 - \frac{2\max\prth{\p,\q}\max\prth{\ell^{\p},\ell'^{\q}}m^{\min\prth{\p,\q}}}{N\min\prth{\p,\q}-2}\Prth{1-\frac{1}{\prth{1+t}^{\prth{N\min\prth{\p,\q}/2}-1}}}}_{\text{Call that }G_{1}\prth{t}}\Bigg]^{-1/\max\prth{\p,\q}}$
}-\reff{a6inf}{%
$\inf\limits_{t\geq 0} \crochets{G_{1}\prth{t}} = 1-\frac{2\max\prth{\p,\q}\max\prth{\ell^{\p},\ell'^{\q}}m^{\min\prth{\p,\q}}}{N\min\prth{\p,\q}-2}.$
}
and\renewcommand{\gap}{-3.5mm}
\reff{CONTROL_uv_decay_rate_diff}{%
$\begin{array}{ll}
	\vertii{\u\prth{t}}{L^{\infty}\prth{\mathbb{R}^{N}}}
	\leq
	\ell
	&\hspace{\gap}\Big(
	\vertii{u_{0}}{L^{1}\prth{\mathbb{R}^{N}}}\skp+\skp
	\vertii{v_{0}}{L^{1}\prth{\mathbb{R}^{N}}}\skp+\skp
	\vertii{\?\?\?\?\?\hat{u}_{0}}{L^{1}\prth{\mathbb{R}^{N}}}\skp+\skp
	\vertii{\?\?\?\?\hat{v}_{\skp\skp\skp 0}}{L^{1}\prth{\mathbb{R}^{N}}}\left.\Big)\right/\prth{1+t}^{N/2},\\[2.5mm] 
	\vertii{\v\prth{t}}{L^{\infty}\prth{\mathbb{R}^{N}}}
	\leq
	\ell'
	&\hspace{\gap}\Big(
	\vertii{u_{0}}{L^{1}\prth{\mathbb{R}^{N}}}\skp+\skp
	\vertii{v_{0}}{L^{1}\prth{\mathbb{R}^{N}}}\skp+\skp
	\vertii{\?\?\?\?\?\hat{u}_{0}}{L^{1}\prth{\mathbb{R}^{N}}}\skp+\skp
	\vertii{\?\?\?\?\hat{v}_{\skp\skp\skp 0}}{L^{1}\prth{\mathbb{R}^{N}}}\left.\Big)\right/\prth{1+t}^{N/2},\\[2.5mm] 
	\text{for all }t>0.&~
\end{array}$
}
in Corollary \ref{CORO_decay_linear_Heat_exchanger} which leads us to take
$$
M : = \Prth{\frac{\prth{N\min\prth{\p,\q}-2}\prth{\ell m}^{\max\prth{\p,\q}}}{N\min\prth{\p,\q}-2-2\max\prth{\p,\q}\max\prth{\ell^{\p},\ell'^{\q}}m^{\min\prth{\p,\q}}}}^{1/\max\prth{\p,\q}}
$$
and
$$
M' : = \Prth{\frac{\prth{N\min\prth{\p,\q}-2}\prth{\ell' m}^{\max\prth{\p,\q}}}{N\min\prth{\p,\q}-2-2\max\prth{\p,\q}\max\prth{\ell^{\p},\ell'^{\q}}m^{\min\prth{\p,\q}}}}^{1/\max\prth{\p,\q}}
$$
which concludes the proof of this second case.
\end{proof}

\section{Systematic blow-up}\label{S5_systematic_blow_up}

In this section, we tackle the proof of Theorem \ref{TH_systematic_BU} which states the blowing-up
of the non-negative solutions to
\reff{SYS_heat_exchanger}{%
\renewcommand{\gap}{2.155mm}
$
\left\lbrace \begin{array}{l}
	\partial_{t}u = c\Delta u - \mu u + \nu v + \hspace{\gap}u^{1+\p},\\
	\partial_{t}v = d\Delta v + \mu u - \nu v + \kappa v^{1+\q},\\
	\prth{u,v}|_{t=0} = \prth{u_{0},v_{0}}.
\end{array} \right.
$
}
when at least one of the two $1/\p$ and $\kappa/\q$ is greater than $N/2$.

We cover both cases of $\kappa=0$ and $\kappa=1$ even though the first one would suffice. To do this we need the slight hypothesis $\p<2/N$ which is transparent for the theorem assumptions, up to put $\kappa$ before $u^{1+\p}$ rather than $v^{1+\q}$ whenever $\q<2/N\leq \p$. Notice that $u$ and $v$ exchange their roles in this case.

Our method involves passing the solution
$\prth{u,v}$ through a Gaussian blur whose intensity is adjusted via a positive parameter $\varepsilon$.
As $\varepsilon$ decreases, the blur increases.
We denote the resulting blurred solution observed at point $x=0$ as $\prth{\pazocalU,\pazocal{V}}$.
We first observe that
the blowing-up of $\prth{\pazocalU,\pazocal{V}}$ implies that of
$\prth{u,v}$.
Then, we show that
$\prth{\pazocalU,\pazocal{V}}$ satisfies an ODI system
which ensures its blowing-up if $\varepsilon$ is chosen sufficiently small.

\bigskip

\begin{proof}[Proof of Theorem \ref{TH_systematic_BU}]
For technical reasons we start by adjusting the datum $\prth{u_{0},v_{0}}$
to give it a specific shape.
This step is non-limiting, thanks to the comparison principle.
First, note that (by shifting time if necessary) we can assume that $u_{0}$ and $v_{0}$ are positive.
As a consequence, there exists small enough $\eta, R>0$ such that
$$
\eta\indicatrice{\pazocal{B}\prth{0,R}}\leq u_{0}
\qquad 
\text{and}
\qquad 
\frac{\mu}{2\nu}\times\eta\indicatrice{\pazocal{B}\prth{0,R}}\leq v_{0}
$$
almost everywhere in $\mathbb{R}^{N}$. Therefore, we can assume without loss of generality that the datum
$\prth{u_{0},v_{0}}$ takes the form
\begin{equation}\label{DEF_shaped_data}
\prth{u_{0},v_{0}} \equiv 
\Prth{\eta\indicatrice{\pazocal{B}\prth{0,R}},\frac{\mu}{2\nu}\times\eta\indicatrice{\pazocal{B}\prth{0,R}}}.
\end{equation}

For $\varepsilon>0$, we define the family of Gaussian kernels $\prth{\Phi_{\varepsilon}}_{\varepsilon>0}$ for all $x\in \mathbb{R}^{N}$ as
\begin{equation}\label{DEF_Gaussian_blur}
\Phi_{\varepsilon}\prth{x} : =  C\prth{\varepsilon}\times e^{-\varepsilon\verti{x}^{2}},
\end{equation}
where $C\prth{\varepsilon} : = \prth{\varepsilon/\pi}^{N/2}$ ensures that $\vertii{\Phi_{\varepsilon}}{L^{1}\prth{\mathbb{R}^{N}}}=1$ for any $\varepsilon>0$. By computing for $\lambda>0$,
\begin{align*}
 	\prth{\Delta+\lambda}\Phi_{\varepsilon}
	&= \prth{4\varepsilon^{2}\verti{x}^{2}-2\varepsilon N+\lambda}\Phi_{\varepsilon}\\
 	&\geq \prth{-2\varepsilon N+\lambda}\Phi_{\varepsilon},
\end{align*}
it becomes evident that
\begin{equation}\label{DEF_sous_fonction_propre}
\Delta\Phi_{\varepsilon}\geq -\lambda\Phi_{\varepsilon}
\end{equation}
if we choose
$\lambda = 2N\varepsilon$, an equality that is maintained throughout the proof.
Now we use $\Phi_{\varepsilon}$ to blur the solution $\prth{u,v}$ by convolution, and we denote the result evaluated at time $t$ and point $x=0$ as $\prth{\pazocalU,\pazocal{V}}$.
More precisely,
\renewcommand{\gap}{-3mm}
\renewcommand{\gapp}{-2.5mm}
\begin{equation}\label{DEF_blurred_solutions}
\begin{array}{lll}
	\pazocalU = \pazocalU\prth{t} &\hspace{\gap}: =
\crochets{\Phi_{\varepsilon}\ast u\prth{t}}\prth{0} &\hspace{\gapp}=
\int_{\mathbb{R}^{N}}^{}\Phi_{\varepsilon}\prth{z}
u\prth{t,z}dz, \\[2mm] 
	\pazocal{V} = \pazocal{V}\prth{t} &\hspace{\gap}: =
\crochets{\Phi_{\varepsilon}\ast v\prth{t}}\prth{0} &\hspace{\gapp}=
\int_{\mathbb{R}^{N}}^{}\Phi_{\varepsilon}\prth{z}
v\prth{t,z}dz.
\end{array}
\end{equation}
Since $\Phi_{\varepsilon}$ is chosen with unit mass, we have, while $\prth{u,v}$ exists in
$\prth{L^{\infty}\prth{\mathbb{R}^{N}}}^{2}$,
\begin{equation}\label{HOLDER_blow_up}
\pazocalU\prth{t} +
\pazocal{V}\prth{t} \leq
\vertii{u\prth{t}}{L^{\infty}\prth{\mathbb{R}^{N}}}+
\vertii{v\prth{t}}{L^{\infty}\prth{\mathbb{R}^{N}}}.
\end{equation}
As a consequence of this inequality,
the blowing-up of
$\prth{\pazocalU,\pazocal{V}}$
yields that of
$\prth{u,v}$
---
possibly at an earlier time.

The objective is now to make
$\prth{\pazocalU,\pazocal{V}}$
blow-up.
Differentiating $\pazocalU$ with respect to time yields (the parameters dependency is locally dropped for better visibility)
\begin{align*}
 	\pazocalU'
 	&= \Phi_{\varepsilon}\ast\Prth{c\Delta u - \mu u + \nu v + u^{1+\p}}\\
 	&= c\crochets{\Delta\Phi_{\varepsilon}\ast u} - \mu \crochets{\Phi_{\varepsilon}\ast u} + 
 	\nu \crochets{\Phi_{\varepsilon}\ast v} + [\Phi_{\varepsilon}\ast u^{1+\p}\?\?\?\?]\\
 	&\geq -c\lambda\crochets{\Phi_{\varepsilon}\ast u} - \mu \crochets{\Phi_{\varepsilon}\ast u} + 
 	\nu \crochets{\Phi_{\varepsilon}\ast v} + \crochets{\Phi_{\varepsilon}\ast u}^{1+\p}\\
 	&= -\prth{\mu+c\lambda}\pazocalU + \nu \pazocal{V} + \pazocalU^{1+\p},
\end{align*}
where we use an integration by part to go from the first to the second line and
\reff{DEF_sous_fonction_propre}{%
$\Delta\Phi_{\varepsilon}\geq -\lambda\Phi_{\varepsilon}.$
}
along Jensen inequality
to go from the second to the third line.
Applying the same approach to
$\pazocal{V}$ leads us the following ODI system
\begin{equation}\label{ODI_system_U_V_eps}
\left\lbrace \begin{array}{lll}
	\pazocalU' \geq  -\prth{\mu+c\lambda}\pazocalU + \nu \pazocal{V} + \pazocalU^{1+\p}, & \qquad & t>0,\\
	\pazocal{V}' \geq \mu \pazocalU -\prth{\nu+d\lambda} \pazocal{V} + \kappa \pazocal{V}^{1+\q}, & \qquad & t>0.
\end{array} \right.
\end{equation}
Due to its cooperative structure, the system
\reff{ODI_system_U_V_eps}{%
$\left\lbrace \begin{array}{l}
	\pazocalU' =  -\prth{\mu+c\lambda}\pazocalU + \nu \pazocal{V} + \pazocalU^{1+\p},\\
	\pazocal{V}' = \mu \pazocalU -\prth{\nu+d\lambda} \pazocal{V} + \kappa \pazocal{V}^{1+\q}.\\
\end{array} \right.$
}
enjoys the comparison principle.
As a result, $\pazocalU$ and $\pazocal{V}$ are respectively above $U$ and $V$ which solve the associated ODE system
\begin{equation}\label{ODE_system_U_V}
\left\lbrace \begin{array}{lll}
	U' =  -\prth{\mu+c\lambda}U + \nu V + U^{1+\p}, & \qquad & t>0,\\
	V' = \mu U -\prth{\nu+d\lambda} V, & \qquad & t>0,\\
\end{array} \right.
\end{equation}
as long as $\prth{\pazocalU,\pazocal{V}}$ and $\prth{U,V}$ start from the same initial datum, that is
\begin{equation}\label{DEF_data_U0_V0}
\prth{U\prth{0},V\prth{0}} =
\Prth{
\int_{\mathbb{R}^{N}}^{}\Phi_{\varepsilon}\prth{z}
u_{0}\prth{z}dz,
\int_{\mathbb{R}^{N}}^{}\Phi_{\varepsilon}\prth{z}
v_{0}\prth{z}dz
}
= :
\prth{U_{0},V_{0}}.
\end{equation}

From this point, the rest of the demonstration consists in showing that $\prth{U,V}$ blows-up if the blur parameter $\varepsilon$ $\prth{=\lambda/2N}$ is chosen sufficiently small. In Figure \ref{FIG_plan_de_phases} below, we present the phase plane associated with ODE system
\reff{ODE_system_U_V}{%
$\left\lbrace \begin{array}{l}
	U' =  -\prth{\mu+c\lambda}U + \nu V + U^{1+\p},\\
	V' = \mu U -\prth{\nu+d\lambda} V.\\
\end{array} \right.$
},
and we take this opportunity to introduce some notations that can be understood by looking at Figure \ref{FIG_plan_de_phases}. We begin by finding the isocline curves for system
\reff{ODE_system_U_V}{%
$\left\lbrace \begin{array}{l}
	U' =  -\prth{\mu+c\lambda}U + \nu V + U^{1+\p},\\
	V' = \mu U -\prth{\nu+d\lambda} V.\\
\end{array} \right.$
}
that are
\begin{equation}\label{DEF_isoclines}
\begin{array}{l}
	\acco{U' = 0} = \Acco{V = \frac{U}{\nu}\prth{\mu+c\lambda-U^{\p}}},
\\[1.5mm] 
	\acco{V' = 0} = \Acco{V = \frac{\mu}{\nu+d\lambda} U}.
\end{array}
\end{equation}
These curves intersect each other at the equilibria
$O : = \prth{0,0}$
and
$E_{1} : = \prth{\chi, \frac{\mu}{\nu+d\lambda}\chi}$, where
\begin{equation}\label{DEF_chi}
\chi := \Prth{\mu+c\lambda-\frac{\mu \nu}{\nu+d\lambda}}^{1/\p}.
\end{equation}
Next, we define
$E_{0} : = \prth{\mu^{1/\p},0}$, and for any $\alfa\in \intervalleff{0}{1}$,
$$
E_{\alfa} : =
E_{0} + \alfa \, \overrightarrow{E_{0}E_{1}} =
\Prth{\mu^{1/\p}\prth{1-\alfa}+\alfa\chi , \; \frac{\mu}{\nu+d\lambda}\alfa\chi}.
$$
We consider then the open set $\Omega$ which is the region above $\acco{V=0}$, below $\acco{V'=0}$, and on the right-hand side of the line $\prth{E_{0}E_{1}}$. More precisely,
\begin{equation}\label{DEF_Omega}
\Omega : =
\Acco{V>0}\cap
\Acco{V<\frac{\mu}{\nu+d\lambda} U} \cap
\Acco{V>\frac{\mu\chi}{\prth{\nu+d\lambda}\prth{\chi-\mu^{1/\p}}} \prth{U-\mu^{1/\p}}}.
\end{equation}
Finally, we name $\pazocal{P}$ and $\pazocal{Q}$ the two components of vector field associated with ODE system
\reff{ODE_system_U_V}{%
$\left\lbrace \begin{array}{l}
	U' =  -\prth{\mu+c\lambda}U + \nu V + U^{1+\p},\\
	V' = \mu U -\prth{\nu+d\lambda} V.\\
\end{array} \right.$
},
namely,
$$
\begin{array}{l}
	\pazocal{P}\prth{U,V} : = -\prth{\mu+c\lambda}U + \nu V + U^{1+\p},\\[1mm] 
	\pazocal{Q}\prth{U,V} : = \mu U -\prth{\nu+d\lambda} V,
\end{array}
$$
and we denote $\overrightarrow{\gamma_{\alfa}}$ the evaluation of the field $\prth{\pazocal{P},\pazocal{Q}}$ at point $E_{\alfa}$.\\[-1mm]
\refstepcounter{FIGURE}\label{FIG_plan_de_phases}
\begin{center}
\includegraphics[scale=1]{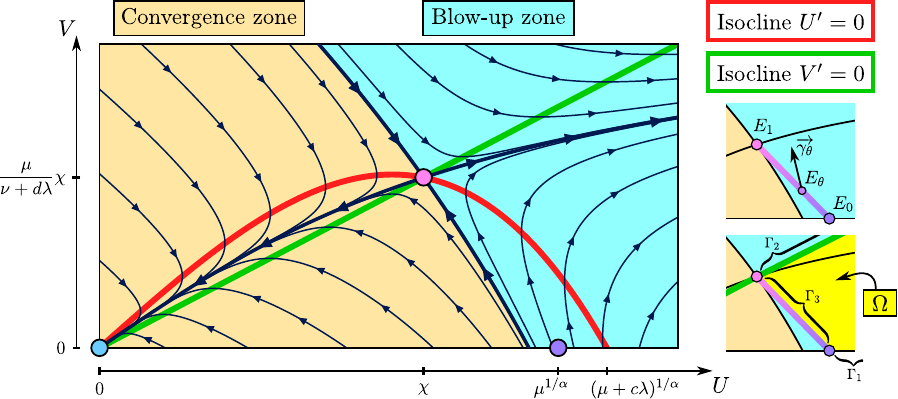}\\[3mm]
\begin{minipage}[c]{14.6079cm}
\textsl{\textbf{Figure \theFIGURE} --- Phase plane associated with ODE system
\reff{ODE_system_U_V}{%
$\left\lbrace \begin{array}{l}
	U' =  -\prth{\mu+c\lambda}U + \nu V + U^{1+\p},\\
	V' = \mu U -\prth{\nu+d\lambda} V.\\
\end{array} \right.$
}.
The equilibria are located at $O = \prth{0,0} = $ \raisebox{-0.35mm}{\includegraphics[scale=1]{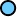}} and $E_{1} = \prth{\chi, \frac{\mu}{\nu+d\lambda}\chi} = $ \raisebox{-0.35mm}{\includegraphics[scale=1]{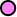}}. The plane splits into two zones depending on whether the trajectories converge to $O$ or blow-up in finite time. As $\lambda$ approaches zero, the equilibrium $E_{1}$ moves towards $O$ until they merge at the limit, and the blow-up zone \glmt{fills} the area below the isocline $\acco{V'=0}$, which is defined by $\acco{V'>0}$.
To characterize the blow-up zone, we define $E_{0} = \prth{\mu^{1/\p}, 0} = $ \raisebox{-0.35mm}{\includegraphics[scale=1]{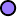}}, and we show that all the points in the region $\Omega$, defined as the intersection of the right-hand side of $\prth{E_{0}E_{1}}$, the upper side of $\acco{V=0}$ and the lower side of $\acco{V'=0}$ are indeed in the blow-up zone provided that $\lambda$ has been sufficiently reduced.}
\end{minipage}
\end{center}

\bigskip

The continuation of our discussion is divided into three steps:
\begin{itemize}[label=\textbullet]
\refstepcounter{ProofTHBUS}\label{ProofTHBUS_i}
	\item [$(\theProofTHBUS)$] We show that by choosing
	$\lambda$ $\prth{=2N\varepsilon}$
	smaller, the region $\Omega$ remains stable under ODE system
	\reff{ODE_system_U_V}{%
$\left\lbrace \begin{array}{l}
	U' =  -\prth{\mu+c\lambda}U + \nu V + U^{1+\p},\\
	V' = \mu U -\prth{\nu+d\lambda} V..
\end{array} \right.$
}.
\refstepcounter{ProofTHBUS}\label{ProofTHBUS_ii}\vspace{-1mm}
	\item [$(\theProofTHBUS)$] We deduce that any solution to
	\reff{ODE_system_U_V}{%
$\left\lbrace \begin{array}{l}
	U' =  -\prth{\mu+c\lambda}U + \nu V + U^{1+\p},\\
	V' = \mu U -\prth{\nu+d\lambda} V.\\
\end{array} \right.$
}
	starting at initial time in $\Omega$ blows-up in finite time.
\refstepcounter{ProofTHBUS}\label{ProofTHBUS_iii}
	\item [$(\theProofTHBUS)$] Finally, we find an $\varepsilon>0$ that places $\prth{U_{0},V_{0}}$ in $\Omega$, leading to the blowing-up of $\prth{U,V}$.
\end{itemize}

\bigskip

\noindent\hspace{-0.5mm}\textbullet\,\textit{Step (\ref{ProofTHBUS_i}).} To demonstrate that $\Omega$ is stable under the ODE system
\reff{ODE_system_U_V}{%
$\left\lbrace \begin{array}{l}
	U' =  -\prth{\mu+c\lambda}U + \nu V + U^{1+\p},\\
	V' = \mu U -\prth{\nu+d\lambda} V.\\
\end{array} \right.$
},
we must establish that the vector field
$\prth{\pazocal{P},\pazocal{Q}}|_{\partial \Omega}$ points inwards $\Omega$. We decompose $\partial \Omega$ into
$\Gamma_{1} \cup \Gamma_{2} \cup \Gamma_{3}$, as illustrated in Figure \ref{FIG_plan_de_phases}:
$$
\partial \Omega = \underbrace{\acco{U>\mu^{1/\p},V=0}}_{\Gamma_{1}} \cup \underbrace{\acco{U>\chi, V'=0}}_{\Gamma_{2}} \cup \underbrace{\crochets{E_{0}E_{1}}}_{\Gamma_{3}}.
$$
We observe that
$\prth{\pazocal{P},\pazocal{Q}}|_{\Gamma_{1}}$
as well as
$\prth{\pazocal{P},\pazocal{Q}}|_{\Gamma_{2}}$
both point in the correct direction since
\begin{itemize}[label=\textbullet]
	\item $\pazocal{Q}>0$ along  
$\Gamma_{1}$, and
	\item $\pazocal{Q}=0<\pazocal{P}$ along
$\Gamma_{2}$.
\end{itemize}
To address the last portion of $\partial\Omega$, which is $\Gamma_{3} = \crochets{E_{0}E_{1}}$, we define the matrices
$$
\pazocal{M}_{\alfa} : =
\begin{pmatrix}
\overrightarrow{\gamma_{\alfa}} & \overrightarrow{E_{0}E_{1}}
\end{pmatrix},
$$
and our goal is to prove $det\prth{\pazocal{M}_{\alfa}}\geq 0$ for all $\alfa\in \intervalleff{0}{1}$,
indicating that
$\overrightarrow{\gamma_{\alfa}}$
points towards the right-hand side of
$\overrightarrow{E_{0}E_{1}}$.
Without going into the details of algebraic computations, we find
\begin{multline*}
\hspace{-2.725mm}det\prth{\pazocal{M}_{\alfa}} =
	\frac{\mu}{\nu+d\lambda}\chi
	\crochets{-\Big(\mu+c\lambda\Big)\Big(\mu^{1/\p}\prth{1-\alfa}+\alfa\chi\Big)
	+\frac{\alfa\mu\nu}{\nu+d\lambda}\chi+\Big(\mu^{1/\p}\prth{1-\alfa}+\alfa\chi\Big)^{1+\p}\,}\\
	+\Big(\mu^{1/\p}-\chi\Big)\mu^{1+1/\p}\Big(1-\alfa\Big),
\end{multline*}
where we can verify that $det\prth{\pazocal{M}_{1}} = 0$, what is consistent with
$\overrightarrow{\gamma_{1}} = \prth{0,0}$
at the equilibrium $E_{1}$.
Now, differentiating $det\prth{\pazocal{M}_{\alfa}}$ with respect to $\alfa$ yields
\begin{multline*}
\hspace{-3,4mm}\partial_{\alfa}\,det\prth{\pazocal{M}_{\alfa}} \!=\!
	\frac{\mu}{\nu+d\lambda}\chi
	\bigg[\Big(\mu+c\lambda\Big)\!\Big(\mu^{1/\p}-\chi\Big)
	+\frac{\mu\nu}{\nu+d\lambda}\chi
	-\Big(1+\p\Big)\!\Big(\mu^{1+\p}-\chi\Big)\!\Big(\mu^{1+\p}-\alfa\prth{\mu^{1+\p}-\chi}\Big)^{\p}\bigg]\\
	-\mu^{1+1/\p}\Big(\mu^{1/\p}-\chi\Big),
\end{multline*}
and by considering that $\chi$, defined in
\reff{DEF_chi}{%
$\chi := \Prth{\mu+c\lambda-\frac{\mu \nu}{\nu+d\lambda}}^{1/\p}.$
},
vanishes as $\lambda$ approaches $0$, we have
$$
\lim\limits_{\lambda \rightarrow 0}
\Prth{
\sup\limits_{\theta\in \intervalleff{0}{1}}
\verti{\partial_{\alfa}\,det\prth{\pazocal{M}_{\alfa}} + \mu^{1+2/\p}}
}
= 0.
$$
From this we deduce that there exists a positive $\lambda_{0}$ such that $\partial_{\alfa}\,det\prth{\pazocal{M}_{\alfa}}<0$ for all
$\lambda\in\intervalleoo{0}{\lambda_{0}}$
and all
$\theta\in \intervalleff{0}{1}$.
This implies
$$
det\prth{\pazocal{M}_{\alfa}} >
det\prth{\pazocal{M}_{1}}=0,
$$
which concludes the step.

\bigskip

\noindent\hspace{-0.5mm}\textbullet\,\textit{Step (\ref{ProofTHBUS_ii}).}
Suppose we have
$\prth{U_{0},V_{0}}\in \Omega$. Then, thanks to step
$(\ref{ProofTHBUS_i})$, we know that
$\prth{U,V}$ remains in $\Omega$ as long as it exists.
Observe that $\pazocal{Q}>0$ in $\Omega$ which causes $V$ to increases.
Since the set
$
\Omega\?\cap\?\acco{U'<0}
$
does not contain any asymptotically stable trajectories, there exists a time
$t_{0}>0$
for which $\prth{U,V}$ crosses the curve $\acco{U'=0}$. After this point,
$$
\prth{U,V} \in \Omega\cap\acco{U'>0},
\qquad
\text{for all }t\in \intervalleoo{t_{0}}{T},
$$
where $T$ denotes the lifetime of $\prth{U,V}$.
From the moment $t_{0}$, $U$ is therefore increasing.
If $U$ were bounded, it would converge, so would $V$ in view of the phase plane, which is impossible.
Consequently, $U$ is unbounded and we can find a time $t_{1}>t_{0}$
that is large enough to perform
$$
-\prth{\mu+c\lambda}U + U^{1+\p} >
\frac{1}{2} U^{1+\p},
\qquad
\text{for all }t\in \intervalleoo{t_{1}}{T}.
$$
Finally, by plugging the latter inequality in first line of
\reff{ODE_system_U_V}{%
$\left\lbrace \begin{array}{l}
	U' =  -\prth{\mu+c\lambda}U + \nu V + U^{1+\p},\\
	V' = \mu U -\prth{\nu+d\lambda} V.\\
\end{array} \right.$
},
we find
$$
U'
>
\frac{1}{2} U^{1+\p} + \nu V
>
\frac{1}{2} U^{1+\p},
\qquad
\text{for all }t\in \intervalleoo{t_{1}}{T},
$$
from which the blowing-up of $U$ is evident.

\bigskip

\noindent\hspace{-0.5mm}\textbullet\,\textit{Step (\ref{ProofTHBUS_iii}).}
The aim of this step is to show that
$\prth{U_{0},V_{0}}$ is in $\Omega$
if we reduce $\varepsilon$ $\prth{=\lambda/2N}$. To achieve this, we must show that
$\prth{U_{0},V_{0}}$
belongs to each set constituting the intersection $\Omega$ in
\reff{DEF_Omega}{%
$\Omega : =
\Acco{V>0}\cap
\Acco{V<\frac{\mu}{\nu+d\lambda} U} \cap
\Acco{V>\frac{\mu\chi}{\prth{\nu+d\lambda}\prth{\chi-\mu^{1/\p}}} \prth{U-\mu^{1/\p}}}.$
}.

At the beginning of the proof, we assume that the datum
$\prth{u_{0},v_{0}}$
takes the form given in
\reff{DEF_shaped_data}{%
$\prth{u_{0},v_{0}} \equiv 
\Prth{\eta\indicatrice{\pazocal{B}\prth{0,R}},\frac{\mu}{2\nu}\times\eta\indicatrice{\pazocal{B}\prth{0,R}}}$
},
where we have
$v_{0} = \frac{\mu}{2\nu}u_{0}$.
Observing the definition of $\prth{U_{0},V_{0}}$ in
\reff{DEF_data_U0_V0}{%
$\prth{U\prth{0},V\prth{0}} =
\Prth{
\int_{\mathbb{R}^{N}}^{}\Phi_{\varepsilon}\prth{z}
u_{0}\prth{z}dz,
\int_{\mathbb{R}^{N}}^{}\Phi_{\varepsilon}\prth{z}
v_{0}\prth{z}dz
}
= :
\prth{U_{0},V_{0}}.$
}
reveals that
$V_{0} = \frac{\mu}{2\nu}U_{0}$ as well.
Therefore, on the phase plane of
Figure \ref{FIG_plan_de_phases},
the datum $\prth{U_{0},V_{0}}$ is located on the line
$\acco{V = \frac{\mu}{2\nu}U}$ which lies below the isocline $\acco{V' = 0}$ if $\lambda$ is chosen sufficiently small --- check
\reff{DEF_isoclines}{%
$\begin{array}{l}
	\acco{U' = 0} = \Acco{V = \frac{U}{\nu}\prth{\mu+c\lambda-U^{\p}}},
\\[1.5mm] 
	\acco{V' = 0} = \Acco{V = \frac{\mu}{\nu+d\lambda} U}.
\end{array}$
}
for confirmation. As a result, there exists $\lambda_{1}>0$ such that
\begin{equation}\label{BEING_IN_OMEGA_1}
\prth{U_{0},V_{0}}\in
\Acco{V<\frac{\mu}{\nu+d\lambda} U},
\qquad
\text{for all }\lambda<\lambda_{1}.
\end{equation}

Next, to place $\prth{U_{0},V_{0}}$ on the right-hand side of
$\prth{E_{0}E_{1}}$,
we must ensure that
\begin{equation}\label{PROOF_BU_1}
V_{0}>\frac{\mu\chi}{\prth{\nu+d\lambda}\prth{\chi-\mu^{1/\p}}} \prth{U_{0}-\mu^{1/\p}}.
\end{equation}
Since $\chi$ vanishes as $\lambda$ approaches to $0$, we may check that the quantity
$\frac{\mu\chi}{\prth{\nu+d\lambda}\prth{\chi-\mu^{1/\p}}}$
is negative if $\lambda$ is sufficiently small. Therefore,
\reff{PROOF_BU_1}{%
$V_{0}>\frac{\mu\chi}{\prth{\nu+d\lambda}\prth{\chi-\mu^{1/\p}}} \prth{U_{0}-\mu^{1/\p}}.$
}
is satisfied if
$$
V_{0}>\frac{\mu^{1+1/\p}\chi}{\prth{\nu+d\lambda}\prth{\mu^{1/\p}-\chi}}.
$$
We can also remove $d\lambda$ and simply require
\begin{align*}	
	V_{0} &>\frac{\mu^{1+1/\p}\chi}{\nu\prth{\mu^{1/\p}-\chi}}\\[2mm]
	&=\frac{\mu}{\nu}\chi + o\prth{\chi}.
\end{align*}
By further reducing $\lambda$ (and consequently $\chi$), it is thus sufficient to have
\begin{align*}
 	V_{0} &> \frac{2\mu}{\nu}\chi\\
	&= \frac{2\mu}{\nu}
\crochets{\mu+c\lambda-\frac{\mu \nu}{\nu+d\lambda}}^{1/\p}\\
	&= \frac{2\mu}{\nu}
\crochets{\Prth{c+\frac{d\mu}{\nu}}\lambda + o\prth{\lambda}}^{1/\p}
\end{align*}
which we can simplified in
\begin{equation}\label{PROOF_BU_2}
V_{0} >
\underbrace{\frac{2^{1+1/\p}\mu}{\nu}\Prth{c+\frac{d\mu}{\nu}}^{1/\p}}_{= : h} \lambda^{1/\p}.
\end{equation}
Returning to the definition of $V_{0}$ in
\reff{DEF_data_U0_V0}{%
$\prth{U\prth{0},V\prth{0}} =
\Prth{
\int_{\mathbb{R}^{N}}^{}\Phi_{\varepsilon}\prth{z}
u_{0}\prth{z}dz,
\int_{\mathbb{R}^{N}}^{}\Phi_{\varepsilon}\prth{z}
v_{0}\prth{z}dz
}
= :
\prth{U_{0},V_{0}}.$
}
with $v_{0}$ in
\reff{DEF_shaped_data}{%
$\prth{u_{0},v_{0}} \equiv 
\Prth{\eta\indicatrice{\pazocal{B}\prth{0,R}},\frac{\mu}{2\nu}\times\eta\indicatrice{\pazocal{B}\prth{0,R}}}$
}
and $\Phi_{\varepsilon}$ in
\reff{DEF_Gaussian_blur}{%
$\Phi_{\varepsilon}\prth{x} : =  C\prth{\varepsilon}\times e^{-\varepsilon\verti{x}^{2}}.$
},
the latter inequality
\reff{PROOF_BU_2}{%
$V_{0} >
\underbrace{\frac{2^{1+1/\p}\mu}{\nu}\Prth{c+\frac{d\mu}{\nu}}^{1/\p}}_{= : h} \lambda^{1/\p}$
}
can be expressed as
$$
\Prth{\frac{\varepsilon}{\pi}}^{N/2}
\frac{\eta\mu}{2\nu}
\int_{\pazocal{B}\prth{0,R}}^{}
e^{-\varepsilon\verti{z}^{2}}
dz >
h\Prth{2N\varepsilon}^{1/\p},
$$
or with an adequate constant $\widetilde{h}$ that depends on
$N,c,d,\mu,\nu,\p,\eta$, 
\begin{equation}\label{PROOF_BU_3}
\int_{\pazocal{B}\prth{0,R}}^{}
e^{-\varepsilon\verti{z}^{2}}
dz >
\widetilde{h} \times \varepsilon^{\prth{1/\p}-\prth{N/2}}.
\end{equation}
Now as $\varepsilon$ approaches $0$, the left-hand side of
\reff{PROOF_BU_3}{%
$\int_{\pazocal{B}\prth{0,R}}^{}
e^{-\varepsilon\verti{z}^{2}}
dz >
\widetilde{h} \times \varepsilon^{\prth{1/\p}-\prth{N/2}}.$
}
converges towards the Lebesgue measure of the ball $\pazocal{B}\prth{0,R}$ that is positive.
Meanwhile, on the right-hand side, 
$\varepsilon^{\prth{1/\p}-\prth{N/2}}$ collapses since $\p$ has been chosen smaller than $2/N$. As a result, by taking
$\varepsilon$ $\prth{=\lambda/2N}$
sufficiently small, we can achieve
\reff{PROOF_BU_3}{%
$\int_{\pazocal{B}\prth{0,R}}^{}
e^{-\varepsilon\verti{z}^{2}}
dz >
\widetilde{h} \times \varepsilon^{\prth{1/\p}-\prth{N/2}}.$
}
and thus 
\reff{PROOF_BU_1}{%
$V_{0}>\frac{\mu\chi}{\prth{\nu+d\lambda}\prth{\chi-\mu^{1/\p}}} \prth{U_{0}-\mu^{1/\p}}.$
},
meaning there exists $\lambda_{2}>0$ such that
\begin{equation}\label{BEING_IN_OMEGA_2}
\prth{U_{0},V_{0}}\in
\Acco{V>\frac{\mu\chi}{\prth{\nu+d\lambda}\prth{\chi-\mu^{1/\p}}} \prth{U-\mu^{1/\p}}},
\qquad
\text{for all }\lambda\in\intervalleoo{0}{\lambda_{2}}.
\end{equation}

Lastly, noting the positivity of $V_{0}$ and considering
\reff{DEF_Omega}{%
$\Omega : =
\Acco{V>0}\cap
\Acco{V<\frac{\mu}{\nu+d\lambda} U} \cap
\Acco{V>\frac{\mu\chi}{\prth{\nu+d\lambda}\prth{\chi-\mu^{1/\p}}} \prth{U-\mu^{1/\p}}}.$
},
\reff{BEING_IN_OMEGA_1}{%
$\prth{U_{0},V_{0}}\in
\Acco{V<\frac{\mu}{\nu+d\lambda} U},
\qquad
\text{for all }\lambda<\lambda_{1}.$
}
and
\reff{BEING_IN_OMEGA_2}{%
$\prth{U_{0},V_{0}}\in
\Acco{V>\frac{\mu\chi}{\prth{\nu+d\lambda}\prth{\chi-\mu^{1/\p}}} \prth{U-\mu^{1/\p}}},
\qquad
\text{for all }\lambda\in\intervalleoo{0}{\lambda_{2}}.$
},
we can conclude that there exists a positive $\lambda$ $\prth{=2N\varepsilon}$ for which
$\prth{U_{0},V_{0}}\in\Omega$.

\bigskip

Gathering the elements of the proof from steps
(\ref{ProofTHBUS_i}),
(\ref{ProofTHBUS_ii}) and
(\ref{ProofTHBUS_iii}),
we demonstrated that we can find a positive $\varepsilon$ for which $\prth{U,V}$ blows-up in finite time, consequently inducing the blowing-up of
$\prth{\pazocalU,\pazocal{V}}$
and
$\prth{u,v}$.
This concludes the proof.
\end{proof}

\bigskip

\noindent\textbf{Acknowledgement.}
The author would like to acknowledge his supervisor
\href{https://alfaro.perso.math.cnrs.fr/}{Matthieu Alfaro}
for his advices and his reviews,
\href{https://www.math.univ-paris13.fr/~souplet/}{Philippe Souplet}
for pointing out the references
\citeee{CuiGlobal98}, 
\citeee{SoupletOptimal04}, 
\citeee{CastilloCritical15}, 
the \href{https://lmrs.univ-rouen.fr/en}{\textit{Laboratoire de Mathématiques Raphaël Salem}} for the Maple\texttrademark ~licence which helped to investigate around linear system
\reff{ODE_Diffusion_Fourier_side}{%
$\partial_{t} \begin{pmatrix} \,\hat{\u}\, \\ \,\hat{\v}\, \end{pmatrix} =
\begin{pmatrix}
-c\verti{\xi}^{2} - \mu & \nu \\ 
\mu & -d\verti{\xi}^{2} - \nu
\end{pmatrix}
\begin{pmatrix} \,\hat{\u}\, \\ \,\hat{\v}\, \end{pmatrix}.$
},
and the {\it Région Normandie} for the financial support of his PhD.

\printbibliography


\newpage
\thispagestyle{empty}

\end{document}